\documentclass[11pt]{article}

\usepackage[centertags]{amsmath}
\usepackage{amsfonts}
\usepackage{amssymb}
\usepackage{amsthm}
\usepackage{newlfont}
\usepackage{bibspacing}

\usepackage[a4paper=true,ps2pdf=false,pagebackref=false]{hyperref}

\hypersetup{
    colorlinks = true,
    linkcolor = magenta,
    anchorcolor = red,
    citecolor = blue,
    filecolor = red,
    pagecolor = red,
    urlcolor = red
}

\newtheorem{thm}{Theorem}[section]
\newtheorem*{thm*}{Theorem}
\newtheorem{cor}[thm]{Corollary}

\newtheorem{lem}[thm]{Lemma}

\newtheorem{prop}[thm]{Proposition}
\newtheorem*{prop*}{Proposition}

\newtheorem*{conj*}{Conjecture}

\newtheorem*{dfn*}{Definition}
\theoremstyle{definition}
\newtheorem{rem}[thm]{\textbf{Remark}}
\newtheorem*{rmk*}{Remark}
\newtheorem*{fact*}{Fact}

\theoremstyle{proof}

\newcommand{\norm}[1]{\left\Vert#1\right\Vert}
\newcommand{\snorm}[1]{\Vert#1\Vert}
\newcommand{\abs}[1]{\left\vert#1\right\vert}
\newcommand{\set}[1]{\left\{#1\right\}}
\newcommand{\brac}[1]{\left(#1\right)}

\newcommand{\Real}{\mathbb{R}}

\newcommand{\eps}{\varepsilon}

\newcommand{\K}{\mathcal{K}}
\newcommand{\I}{\mathcal{I}}
\newcommand{\M}{\mathcal{M}}
\renewcommand{\L}{\mathcal{L}}

\newlength{\defbaselineskip}
\numberwithin{equation}{section}

\def \F {\mathcal{F}}

\begin{document}

\title{Properties of Isoperimetric, Functional and Transport-Entropy Inequalities Via Concentration
}
\markright{Isoperimetric Properties Via Concentration}
\author{Emanuel Milman\textsuperscript{1}}

\footnotetext[1]{
Department of Mathematics,
University of Toronto,
40 St. George Street,
Toronto, Ontario M5S 2E4,
Canada.
Email: emilman@math.toronto.edu.\\
2000 Mathematics Subject Classification: 60E15, 46G12, 60B99.\\
Keywords: isoperimetric inequality; log-Sobolev inequality; Transport-Entropy inequality; concentration inequality; stability under perturbation; Wasserstein distance.}

\maketitle

\begin{abstract}
Various properties of isoperimetric, functional, Transport-Entropy and concentration inequalities are studied on a Riemannian manifold equipped with a measure, whose generalized Ricci curvature is bounded from below. First, stability of these inequalities with respect to perturbation of the measure is obtained. The extent of the perturbation is measured using several different distances between perturbed and original measure, such as a one-sided $L^\infty$ bound on the ratio between their densities, Wasserstein distances, and Kullback--Leibler divergence. In particular, an extension of the Holley--Stroock perturbation lemma for the log-Sobolev inequality is obtained, and the dependence on the perturbation parameter is improved from linear to logarithmic. Second, the equivalence of Transport-Entropy inequalities with different cost-functions is verified, by obtaining a reverse Jensen type inequality. The main tool used is a previous precise result on the equivalence between concentration and isoperimetric inequalities in the described setting. Of independent interest is a new dimension independent characterization of Transport-Entropy inequalities with respect to the $1$-Wasserstein distance, which does not assume any curvature lower bound.
\end{abstract}

\pagestyle{myheadings}

\section{Introduction}

Let $(\Omega,d)$ denote a complete separable metric space, and let $\mu$
denote a Borel probability measure on $(\Omega,d)$.
One way to measure the interplay between the metric $d$ and the measure $\mu$ is by means of an isoperimetric
inequality. Recall that Minkowski's (exterior) boundary measure of a
Borel set $A \subset \Omega$, which we denote here by $\mu^+(A)$, is
defined as $\mu^+(A) := \liminf_{\eps \to 0} \frac{\mu(A^d_{\eps}) -
\mu(A)}{\eps}$, where $A_{\eps}=A^d_{\eps} := \set{x \in \Omega ; \exists y
\in A \;\; d(x,y) < \eps}$ denotes the $\eps$ extension of $A$ with
respect to the metric $d$. The isoperimetric profile $\I =
\I_{(\Omega,d,\mu)}$ is defined as the pointwise maximal function $\I
: [0,1] \rightarrow \Real_+$, so that $\mu^+(A) \geq \I(\mu(A))$, for
all Borel sets $A \subset \Omega$. An isoperimetric inequality measures the relation between the boundary measure and the measure of a set, by providing a lower bound on $\I_{(\Omega,d,\mu)}$ by some function $J: [0,1] \rightarrow \Real_+$ which is not identically $0$. Since $A$ and $\Omega \setminus A$ will typically have the same boundary measure, it will be convenient to also define $\tilde{\I} :[0,1/2] \rightarrow \Real_+$ as $\tilde{\I}(v) := \min(\I(v),\I(1-v))$.

Another way to measure the relation between $d$ and $\mu$ is given by concentration inequalities. The log-concentration profile $\K = \K_{(\Omega,d,\mu)}$ is defined as the pointwise maximal function $\K:\Real_+ \rightarrow \Real$ such that $1 - \mu(A^d_r) \leq \exp(-\K(r))$ for all Borel sets $A \subset \Omega$ with $\mu(A) \geq 1/2$. Note that $\K(r) \geq \log 2$ for all $r \geq 0$. Concentration inequalities measure how tightly the measure $\mu$ is concentrated around sets having measure $1/2$ as a function of the distance $r$ away from these sets, by providing a lower bound on $\K$ by some non-decreasing function $\alpha: \Real_+ \rightarrow \Real \cup \set{+\infty}$, so that $\alpha$ tends to infinity.
The two main differences between isoperimetric and concentration inequalities are that the latter ones only measure the concentration around sets having measure $1/2$, and do not provide any information for small distances $r$ (smaller than $\alpha^{-1}(\log 2)$). We refer to \cite{Ledoux-Book,MilmanSurveyOnConcentrationOfMeasure1988} for a wider exposition on these and related topics and for various applications.

It is known and easy to see that an isoperimetric inequality always implies a concentration inequality, simply by ``integrating'' along the isoperimetric differential inequality (see Section \ref{sec:pre}). In fact, it will be useful to also consider other \emph{intermediate} levels between these two extremes, such as functional inequalities (e.g. Poincar\'e, Sobolev and log-Sobolev inequalities) and Transport-Entropy inequalities (e.g. Talagrand's $T_2$ inequality and its various variants). These will be introduced later on, but for now, let us just mention that it is known that these types of inequalities typically follow from appropriate isoperimetric inequalities, and imply appropriate concentration inequalities (see \cite{Ledoux-Book,VillaniOldAndNew} or Section \ref{sec:pre} and the references therein). Schematically, this can represented in the following hierarchical diagram:
\begin{multline} \label{eq:intro-hierarchy}
\text{Isoperimetric inequalities } \Rightarrow \text{ Functional inequalities } \\
\Rightarrow \text{ Transport-Entropy inequalities } \Rightarrow \text{ Concentration inequalities} ~.
\end{multline}

All of the converse statements to the implications above are in general known to be false, due to the
possible existence of narrow ``necks'' in the geometry of the space $(\Omega,d)$ or the measures $\mu$. However, when such necks are ruled out by imposing some semi-convexity assumptions on the geometry and measure in the Riemannian-manifold-with-density setting (defined below), it was shown in our previous work \cite{EMilmanGeometricApproachCRAS,EMilmanGeometricApproachPartI} that isoperimetric and concentration inequalities are in fact equivalent, with quantitative estimates which \emph{do not depend} on the dimension of the underlying manifold (see Section \ref{sec:pre} for a precise formulation).

\medskip

The main purpose of this work is to obtain several new applications of this equivalence between isoperimetry and concentration (and hence of all the intermediate levels as well), which seem to have been previously inaccessible.

\subsection{Setup}

We will henceforth assume that $\Omega$ is a smooth complete oriented connected $n$-dimensional ($n\geq 2$) Riemannian manifold
$(M,g)$, that $d$ is the induced geodesic distance, and that $\mu$ is an absolutely continuous measure with respect to the Riemannian volume form $vol_M$ on $M$.

\begin{dfn*} \label{def:CA}
We will say that our \emph{smooth $\kappa$-semi-convexity assumptions} are satisfied $(\kappa \geq 0)$ if $\mu$ is supported on a geodesically convex set $S \subset M$, on which $d\mu = \exp(-\psi) dvol_M|_S$ with $\psi \in C^2(S)$, and as tensor fields on $S$:
\[
 Ric_g + Hess_g \psi \geq -\kappa g ~.
\]
We will say that our \emph{$\kappa$-semi-convexity assumptions} are satisfied if $\mu$ can be approximated in total-variation by measures $\set{\mu_m}$ so that each $(\Omega,d,\mu_m)$ satisfies our smooth $\kappa$-semi-convexity assumptions. \\
When $\kappa=0$, we will say in either case that our \emph{(smooth) convexity assumptions} are satisfied.
\end{dfn*}

\smallskip

Here $Ric_g$ denotes the Ricci curvature tensor of $(M,g)$ and $Hess_g$ denotes the second covariant derivative. $Ric_g + Hess_g \psi$ is the well-known Bakry--\'Emery curvature tensor, introduced in \cite{BakryEmery} (in the more abstract framework of diffusion generators), which incorporates the curvature from both the geometry of $(M,g)$ and the measure $\mu$. When $\psi$ is sufficiently smooth and $S=M$, our $\kappa$-semi-convexity assumption is then precisely the Curvature-Dimension condition $CD(-\kappa,\infty)$ (see \cite{BakryEmery}). An important example to keep in mind is that of Euclidean space $(\Real^n,\abs{\cdot})$ equipped with a probability measure $\exp(-\psi(x)) dx$ with $Hess \; \psi \geq -\kappa Id$.

\subsection{Stability of isoperimetric and functional inequalities}

One central theme in this work will be in deducing new stability results of isoperimetric and functional inequalities with respect to perturbations of the measure $\mu$, in the presence of our semi-convexity assumptions. More precisely, if $\mu_1,\mu_2$ are two probability measures on $(M,g)$ such that $\mu_2$ is close to $\mu_1$ with respect to some (not necessarily symmetric) distance, we will show that under appropriate semi-convexity assumptions, $(M,g,\mu_2)$ inherits from $(M,g,\mu_1)$ a quantitatively comparable isoperimetric or functional inequality. A-priori, it seems very difficult to analyze the stability of these questions directly (at least for non-trivial distances), possibly due to the fact that in general no stability is possible and that
some further weak convexity conditions need to be imposed. Our approach for obtaining stability results in such cases, is to
decouple the stability question from the convexity assumptions. We first pass from the isoperimetric or functional inequality to an appropriate concentration inequality (this is always possible without any further assumptions); the stability question on the level of concentration turns out to be elementary, and it is easy to obtain a concentration inequality for the perturbed measure; lastly, we utilize our semi-convexity assumptions and pass back to the isoperimetric or functional level, by employing the equivalence between concentration and isoperimetry.

Our results apply to several different notions of distance between $\mu_1$ (the original measure) and $\mu_2$ (the perturbed measure):
\begin{enumerate}
\item
$\snorm{\frac{d\mu_2}{d\mu_1}}_{L^\infty} \leq D$ - we show in Section \ref{sec:stab1} that when $\mu_2$ satisfies our semi-convexity assumptions, and $\mu_1$ satisfies a strong enough isoperimetric or functional inequality, the latter property is inherited by $\mu_2$. We formulate some results specifically for the log-Sobolev inequality, since it lies precisely on the border of our method, and since it is very useful in applications. Our results extend beyond the classical stability result for the log-Sobolev inequality of Holley and Stroock \cite{HolleyStroockPerturbationLemma}; in addition, under our convexity assumptions, we improve the classical quantitative dependence on $D$ from linear to logarithmic.
\item
$\widetilde{W}_{\Psi_1}(\mu_1,\mu_2) \leq D$ - we introduce a new metric $\widetilde{W}_{\Psi_1}$ between exponentially integrable probability measures, which we call the $\Psi_1$-Lipschitz metric. We show in Section \ref{sec:stab2} analogous stability results for this distance as for the first one, when $\mu_2$ satisfies our semi-convexity assumptions. The main advantage of using this distance over the first one is that this result may be applicable even when the measures $\mu_1,\mu_2$ are mutually singular.
\item
$W_{1}(\mu_1,\mu_2) \leq D$ - using the usual $1$-Wasserstein distance $W_1$, we show in Section \ref{sec:stab2} that under our convexity assumptions ($\kappa = 0$ case), linear (Cheeger type) isoperimetric inequalities or Poincar\'e inequalities are easily inherited. This also applies when having control over the relative entropies: $H(\mu_2 | \mu_1) \leq D$ or $H(\mu_1 | \mu_2) \leq D$.
\end{enumerate}

The importance of obtaining dimension-free log-Sobolev and Poincar\'e inequalities in the context of Statistical Mechanics has been clarified in the works of Stroock--Zegarlinski \cite{StroockZegarlinskiLogSobImpliesDSCondition}, Yoshida \cite{YoshidaLogSobEquivToDecayOfCorrelations} and Bodineau--Helffer \cite{BodineauHelfferSurvey}, who showed (roughly speaking) that these are equivalent to the decay of spin-spin correlations (implying in particular the uniqueness of the corresponding Gibbs measure in the thermodynamic limit).
The stability results described above are therefore relevant in understanding the effects of perturbing the Hamiltonian potential in this context.

\subsection{Transport-Entropy Inequalities: characterization via concentration and equivalence for different cost-functions} \label{subsec:intro-TE-equiv}

Another application pertains to Transport-Entropy inequalities, first introduced by Marton \cite{Marton86,Marton96} and developed by Talagrand \cite{TalagrandT2} (see Section \ref{sec:pre}). In general, given two convex functions $\phi,\psi: \Real_+ \rightarrow \Real_+$, one may define a $(\phi,\psi)$ Transport-Entropy inequality as the following statement:
\begin{equation} \label{eq:intro-TE}
\exists D >0 \;\;\; W_{c_{\phi,D}}(\nu,\mu) \leq \psi^{-1}(H(\nu | \mu)) \;\;\; \forall \text{ probability measure } \nu ~,
\end{equation}
where $c_{\phi,D}$ denotes the cost-function $c_{\phi,D}(x,y) := \phi(D d(x,y))$, $W_c$ is the Wasserstein distance with cost-function $c$, and $H(\nu | \mu)$ denotes the relative entropy (see Section \ref{sec:pre} for definitions). By Jensen's inequality, it is immediate that $(\ref{eq:intro-TE})$ implies the following (in general, strictly) weaker statement:
\begin{equation} \label{eq:intro-TE-weak}
\exists D >0 \;\;\; D W_{d}(\nu,\mu) \leq (\psi \circ \phi)^{-1}(H(\nu | \mu)) \;\;\; \forall \text{ probability measure } \nu ~.
\end{equation}
It was shown by Bobkov and G\"{o}tze \cite{BobkovGotzeLogSobolev} that the latter inequality is in fact equivalent to a variant of a concentration inequality. We clarify this equivalence in Proposition \ref{prop:con-TE}, which seems new and may be of independent interest, removing the inherent dimension dependence in previous results by
Djellout--Guillin--Wu \cite{DjelloutGuillinWu}, Bolley--Villani \cite{BolleyVillani}, and Gozlan--Leonard \cite{GozlanLeonard}. Using this characterization, we show that under our semi-convexity assumptions, it is possible in many cases to obtain a converse to Jensen's inequality, up to dimension-independent constants, passing back from (\ref{eq:intro-TE-weak}) to (\ref{eq:intro-TE}).

\subsection{Organization}

It would be very hard to describe our results in more detail in any reasonably sized Introduction, without first defining many necessary notions. These are given in Section \ref{sec:pre}, which serves as an extended introduction to this work. We describe the hierarchy (\ref{eq:intro-hierarchy}) in detail, and formulate
our main result from \cite{EMilmanGeometricApproachCRAS,EMilmanGeometricApproachPartI}, asserting the equivalence of concentration and isoperimetric inequalities under the semi-convexity assumptions, which permits reversing the hierarchy and completing the equivalence. The stability results with respect to perturbation of the measure are given in Sections \ref{sec:stab1} (for the distance $\snorm{\frac{d\mu_2}{d\mu_1}}_{L^\infty}$) and \ref{sec:stab2} (for the $\Psi_1$-Lipschitz, Wasserstein and relative entropy distances). In Section \ref{sec:inter}, we set the ground for properly relating between Transport-Entropy and concentration inequalities; parts of it may be of independent interest. In Section \ref{sec:TE-equiv}, we briefly state the results on the equivalence of Transport-Entropy inequalities with different cost-functions.

\medskip

\noindent \textbf{Acknowledgments.} I would like to thank Cedric Villani, Alexander Kolesnikov and Franck Barthe for their interest and remarks on this work. I also gratefully acknowledge the support of the Institute of Advanced Study, where this work was initiated, and of the University of Toronto (and in particular, of Robert McCann), where this work was concluded. Final thanks go out to the referee, for valuable and effective comments.

\section{Preliminaries} \label{sec:pre}

We reserve the use of $c,c',c_1,c_2,c_3,c_1',c_2',c_3',C,C'$ etc. to indicate universal numeric constants, independent of all other parameters (and in particular of the dimension of any underlying manifold), whose values may change from one occurrence to the next. We also use the notation $A \simeq B$ to signify that there exist constants $c_1,c_2>0$ so that $c_1 B \leq A \leq c_2 B$, and that these constants do not depend on any other parameter, unless explicitly stated otherwise. When $c_1,c_2$ depend on some additional set of parameters $S$, we may also use the notation $A \simeq_{S} B$.

\subsection{Definitions and Notation}

Let us start by recalling some of the notions mentioned in the Introduction.

Let $\F = \F(\Omega,d)$ denote the space of functions which are Lipschitz on every ball in $(\Omega,d)$, and let $f \in \F$. Functional inequalities compare between some type of expression measuring the $\mu$-averaged oscillation of $f$, and an expression measuring the $\mu$-averaged magnitude of the gradient $|\nabla f|:= g(\nabla f,\nabla f)^{1/2}$. In the general metric-space setting, one may define $|\nabla f|$ as the following Borel function:
\[
 \abs{\nabla f}(x) := \limsup_{d(y,x) \rightarrow 0+} \frac{|f(y) - f(x)|}{d(x,y)} ~.
\]
(and we define it as 0 if $x$ is an isolated point - see \cite[pp. 184,189]{BobkovHoudre} for more details).
Some well known examples of functional inequalities include the Poincar\'e and Sobolev-Gagliardo-Nirenberg inequalities, but the one which will be of most interest to us in this work is the \emph{log-Sobolev inequality}, introduced by Gross \cite{GrossIntroducesLogSob} in the study of the Gaussian measure, which corresponds to the case $q=2$ below. The extension to the range $q \in [1,2]$ is due to Bobkov and Zegarlinski \cite{BobkovZegarlinski}.
\begin{dfn*}
$(\Omega,d,\mu)$ satisfies a $q$-log-Sobolev inequality ($q \in [1,2]$) if:
\begin{equation} \label{eq:LS-inq-def}
\exists D >0 \; \text{ s.t. } \; \forall f \in \F \;\;\;\;  D (Ent_\mu(|f|^q))^{1/q} \leq \norm{\abs{\nabla f}}_{L^q(\mu)} ~,
\end{equation}
where $Ent_\mu(g)$ denotes the entropy of a non-negative function $g$:
\[
 Ent_\mu(g) := \int g \log \brac{g / \int g d\mu} d\mu ~.
\]
The best possible constant $D$ above is denoted by $D_{LS_q} = D_{LS_q}(\Omega,d,\mu)$.
\end{dfn*}

Another way to measure the interplay between the metric $d$ and the measure $\mu$ is given by Transport-Entropy (or TE) inequalities, first introduced by Marton \cite{Marton86,Marton96}, and significantly developed by Talagrand \cite{TalagrandT2}. These compare between the cost of optimally transporting between $\mu$ and a second probability measure $\nu$ (with respect to some cost function $c : \Omega \times \Omega \rightarrow \Real_+$), and the relative entropy of $\nu$ with respect to $\mu$. The transportation cost, or Wasserstein distance, is defined as:
\[
W_c(\nu,\mu) := \inf \int_{\Omega \times \Omega} c(x,y)
d\pi(x,y) ~,
\]
where the infimum runs over the set $\M(\nu,\mu)$ of all probability measures $\pi$ on the
product space $\Omega \times \Omega$ with marginals $\nu$ and $\mu$.
We reserve the notation $W_p$ to denote $W_{d^p}^{1/p}$ ($p \geq 1$), which is known (e.g. \cite[Theorem 6.9]{VillaniOldAndNew}) to metrize the appropriate weak topology on the space of Borel probability measures $\mu$ on $(\Omega,d)$ having a finite $p$-th moment: $\int d(x,x_0)^p d\mu(x) < \infty$. Here $x_0$ is some (equivalently, any) fixed point in $\Omega$. The relative entropy, or Kullback--Leibler divergence, is defined for $\nu \ll \mu$ as:
\[
H(\nu | \mu) := Ent_\mu(\frac{d\nu}{d\mu}) = \int \log(\frac{d\nu}{d\mu}) d\nu ~,
\]
and $+\infty$ otherwise. An important example of a Transport-Entropy inequality is given by Talagrand's $T_2$ inequality, corresponding to the case $s=p=2$ below:
\begin{dfn*}
$(\Omega,d,\mu)$ satisfies a $(s,p)$ Transport-Entropy inequality ($p\geq 2, s \geq 1$) if:
\begin{equation} \label{eq:TE-inq-def}
\exists D >0 \; \text{ s.t. } \; \forall \text{ probability measure } \nu \;\;\; D W_s(\nu,\mu) \leq H(\nu | \mu)^{1/p} ~.
\end{equation}
The best possible constant $D$ above is denoted by $D_{TE_{s,p}} = D_{TE_{s,p}}(\Omega,d,\mu)$.
\end{dfn*}

The restriction to $q \in [1,2]$ and $p\geq 2$ above is necessary. Indeed, setting $f = 1+\eps g$ in (\ref{eq:LS-inq-def}) and letting $\eps \rightarrow 0$, one checks that the left-hand-side behaves like $\eps^{2/q}$ whereas the right-hand-side behaves like $\eps$. Similarly, setting $\nu = (1 + \eps g) \mu$ with $\int g d\mu = 0$ in (\ref{eq:TE-inq-def}), the right-hand-side behaves like $\eps^{2/p}$ whereas the left-hand-side behaves like $\eps$. It is however possible to extend these definitions to the range $q \geq 2$ and $p \in [1,2]$ by using appropriate modified log-Sobolev
and TE inequalities with cost function modified to be quadratic for small distances, in the spirit of Talagrand \cite{TalagrandT2}. To describe these variants, let us introduce the convex function $\varphi_p : \Real_+ \rightarrow \Real_+$, which is given by:
\begin{equation} \label{eq:def-varphi}
\varphi_p(x) :=  \frac{x^p}{p}  \;\;\;  \text{ if } p \geq 2 \;\;\; \text{and} \;\;\; \varphi_p(x) := \begin{cases} \frac{x^2}{2} & x \in [0,1] \\ \frac{x^p}{p} + \frac{1}{2} - \frac{1}{p} & x \in (1,\infty) \end{cases} \;\;\; \text{ if } p \in [1,2]  ~.
\end{equation}
Denoting by $\varphi_{*,q}(\lambda) := (\varphi_p)^*(\lambda) := \sup_{x \geq 0} \lambda x - \varphi_p(x)$ the Legendre transform of $\varphi_p$, where $q = p^* := p/(p-1)$, one checks that:
\begin{equation} \label{eq:def-varphi*}
\varphi_{*,q}(\lambda) :=  \frac{\lambda^q}{q}  \;\;\;  \text{ if } q \in [1,2] \;\;\; \text{and} \;\;\; \varphi_{*,q}(\lambda) := \begin{cases} \frac{\lambda^2}{2} & \lambda \in [0,1] \\ \frac{\lambda^q}{q} + \frac{1}{2} - \frac{1}{q} & \lambda \in (1,\infty) \end{cases} \;\;\; \text{ if } q \in [2,\infty]  ~.
\end{equation}

\begin{dfn*}
$(\Omega,d,\mu)$ satisfies a $q$-modified-log-Sobolev inequality ($q \in [1,\infty]$) if:
\begin{equation} \label{eq:mLS-inq-def}
\exists D >0 \; \text{ s.t. } \; \forall f \in \log \F \;\;\;\;  Ent_\mu(f^2) \leq \int f^2 \varphi_{*,q}\brac{\frac{1}{D}\frac{|\nabla f|}{|f|}} d\mu ~.
\end{equation}
The best possible constant $D$ above is denoted by $D_{mLS_q} = D_{mLS_q}(\Omega,d,\mu)$.
\end{dfn*}
Here $\log \F$ denotes the class of functions $f$ such that $\log(f^2) \in \F$, and $|\nabla f|/|f|$ should be understood as $|\nabla \log(f^2)|/2$.
Note that substituting $f = g^{q/2}$ above when $q \in [1,2]$, we see that $q$-modified and $q$-log-Sobolev inequalities coincide, with $D_{LS_q} = q^{1/q} D_{mLS_q}$.
The case $q = \infty$ ($p=1$) was first introduced by Bobkov and Ledoux \cite{BobkovLedouxModLogSobAndPoincare} and further studied by Bobkov--Gentil--Ledoux \cite{BobkovGentilLedoux}. The extension to the entire range $q \geq 2$ is due to Gentil--Guillin--Miclo \cite{GentilGuillinMiclo}.

\begin{dfn*}
$(\Omega,d,\mu)$ satisfies a $(\varphi_p,1)$ Transport-Entropy inequality ($p \geq 1$) if:
\begin{equation} \label{eq:varphi_p-TE-def}
W_{c_{\varphi_p,D}}(\nu, \mu) \leq H(\nu | \mu) \;\;\;  \forall \text{ probability measure } \nu  ~,
\end{equation}
where $c_{\varphi_p,D}$ denotes the cost function $c_{\varphi_p,D}(x,y) := \varphi_{p}(D d(x,y))$. The best possible constant $D$ above is denoted by $D_{TE_{\varphi_p,1}} = D_{TE_{\varphi_p,1}}(\Omega,d,\mu)$.
\end{dfn*}
Again, in the case $p\geq 2$, these are just obviously identical to the usual $(p,p)$ Transport-Entropy inequalities with $D_{TE_{(\varphi_p,1)}} = p^{1/p} D_{TE_{p,p}}$, so the novelty lies in the extension to the range $p \in [1,2]$. The case $p=1$ was first introduced by Talagrand \cite{TalagrandT2} in his study of the exponential measure on $\Real$, and further characterized in \cite{BobkovLedouxModLogSobAndPoincare,BobkovGentilLedoux} (see below). The entire range $p \in [1,2]$ has  been subsequently considered by various authors (see \cite[Chapter 22]{VillaniOldAndNew} and the references therein), and in particular by Gentil--Guillin--Miclo \cite{GentilGuillinMiclo}, who connected them to modified $q$-log-Sobolev inequalities.

\begin{rem}
Due to their non-homogeneous nature, various variants of $(\varphi_p,1)$ TE and $q$-modified-log-Sobolev inequalities have been used in the literature, with corresponding constants $D_{TE}$ and $D_{mLS}$ appearing in front of different terms. We stand behind our convention, since all of the constants $D$ above and throughout this work scale linearly with the metric, i.e. $D(\Omega,\lambda d,\mu) = D(\Omega,d, \mu) / \lambda$. In any case, it is easy to modify these variants into our form, by using the following easy to verify properties of $\varphi_p$:
\begin{eqnarray}
\label{eq:varphi_p-prop} p \in [1,2] & \Rightarrow & c \varphi_p(x) \geq \varphi_p(\min(c,1) x)  \;\;\; \forall x \geq 0  ~ ; \\
\label{eq:varphi*_q-prop} q \in [2,\infty] & \Rightarrow & C \varphi_{*,q}(\lambda) \leq \varphi_{*,q}(\max(\sqrt{C},1) \lambda ) \;\;\; \forall \lambda \geq 0 ~.
\end{eqnarray}
\end{rem}

Before proceeding, we mention a useful characterization of the case $p=1$ ($q = \infty$) obtained by Bobkov--Gentil--Ledoux \cite{BobkovGentilLedoux}: the $(\varphi_1,1)$ TE inequality is equivalent to the $\infty$-modified-log-Sobolev inequality, which in turn is equivalent (\cite{BobkovLedouxModLogSobAndPoincare}) to the well-known Poincar\'e inequality:
\[
\exists D > 0 \;\;\; D^2 \brac{\int f^2 d\mu - (\int f d\mu)^2} \leq \int |\nabla f|^2 d\mu \;\;\; \forall f \in \F ~.
\]
Denoting the best possible constant $D$ above by $D_{Poin} := D_{Poin}(\Omega,d,\mu)$, the equivalence is in the sense that:
\begin{equation} \label{eq:BGL-equiv}
D_{Poin} \simeq D_{mLS_\infty} \simeq D_{TE_{(\varphi_1,1)}} ~.
\end{equation}
We refer to \cite{BobkovLedouxModLogSobAndPoincare,BobkovGentilLedoux} for a more precise statement.

\medskip

As mentioned in the Introduction, it is known that functional and Transport-Entropy inequalities may be used to interpolate between isoperimetric and concentration inequalities. To dispense of unneeded generality, let us illustrate this hierarchy in an important family of examples.

\begin{dfn*}
$(\Omega,d,\mu)$ is said to satisfy a \emph{$p$-exponential isoperimetric inequality}, $p \in [1,\infty)$, if:
\[
\exists D>0 \;\;\; \I_{(\Omega,d,\mu)} \geq D \I_{(\Real,\abs{\cdot},\Gamma_p)} ~,
\]
where $\Gamma_p$ denotes the probability measure on $\Real$ with density $\exp(-|x|^p/p)/Z_p$ (and $Z_p$ is a normalization factor). We denote the best constant $D$ above by $D_{Iso_p}= D_{Iso_p}(\Omega,d,\mu)$. The case $p=2$ corresponds to the standard Gaussian measure on $\Real$, and is called a \emph{Gaussian isoperimetric inequality}.
\end{dfn*}

\begin{dfn*}
$(\Omega,d,\mu)$ is said to satisfy a \emph{$p$-exponential concentration inequality}, if:
\begin{equation} \label{eq:Exp-p}
\exists D>0 \;\;\; \K_{(\Omega,d,\mu)}(r) \geq -1 + (D r)^p \;\;\; \forall r \geq 0 ~.
\end{equation}
We denote the best constant $D$ above by $D_{Con_p} = D_{Con_p}(\Omega,d,\mu)$.
\end{dfn*}
\begin{rem}
The purpose of the $-1$ above is to emphasize that this only provides information on the behaviour of $\K$ in the large, and could be replaced by any constant strictly smaller than $\log 2$.
\end{rem}
\begin{rem} \label{rem:iso_p}
It is known (see \cite{BobkovExtremalHalfSpaces},\cite{BobkovZegarlinski}) that given $p \in [1,\infty)$, $\tilde{\I}_{(\Real,\abs{\cdot},\Gamma_p)}(v) \simeq_p v \log^{1/q} 1/v$ uniformly on $v \in [0,1/2]$, with $q = p^*$  (the lower bound on $\tilde{\I}_{(\Real,\abs{\cdot},\Gamma_p)}$ is universal, but the upper bound will depend on $p$). The space $(\Real,\abs{\cdot},\Gamma_p)$ is a prototype for all of the inequalities mentioned above, and it is known that it satisfies $D_{Con_p}, D_{TE_{\varphi_p,1}} , D_{mLS_q} \simeq_p 1$ for all $p \in [1,\infty)$.
\end{rem}

\subsection{The Hierarchy}

As already mentioned, it is known and easy to see (e.g. \cite[Proposition 1.7]{EMilmanSodinIsoperimetryForULC}) that an isoperimetric inequality always implies a concentration inequality, simply by ``integrating'' along the isoperimetric differential inequality. Namely, if $\gamma : \Real_+ \rightarrow \Real_+$ is a continuous function, then:
\begin{equation} \label{eq:i-c}
\begin{array}{c}
\tilde{\I}(v) \geq v \gamma(\log 1/v) \;\;\; \forall v\in[0,1/2] \\
\Downarrow \\
\K(r) \geq \alpha(r) \;\;\; \forall r \geq 0 \;\;\; \textrm{where} \;\; \alpha^{-1}(x) = \int_{\log 2}^x \frac{dy}{\gamma(y)} ~.
\end{array}
\end{equation}

It is immediate to check that (\ref{eq:i-c}) yields the following implication ($p\geq 1$):
\begin{equation} \label{eq:p-isop-p-conc}
\text{$p$-exponential isoperimetric inequality } \Rightarrow \text{ $p$-exponential concentration inequality} ~.
\end{equation}
Using the same notation as above, the following known series of implications clearly interpolates between these two extremes (for $p \geq 2$ and $q = p^*$):
\begin{multline} \label{eq:hierarchy}
\text{$p$-exponential isoperimetric inequality } \Rightarrow \text{ $q$-log-Sobolev inequality } \Rightarrow \\
\text{$(p,p)$ Transport-Entropy inequality } \Rightarrow \text{ $p$-exponential concentration inequality} ~.
\end{multline}
More precisely, given $p \in [2,\infty)$, there exist constants $C_1,C_2,C_3 > 0$ so that:
\[
D_{Iso_p} \leq C_1 D_{LS_{q}} \leq C_2 D_{TE_{p,p}} \leq C_3 D_{Con_p} ~.
\]
It is also possible to extend the above hierarchy to the range $p \in [1,2)$, but this is slightly less known and requires further explanation:
\begin{multline} \label{eq:hierarchy-2}
\text{$p$-exponential isoperimetric inequality } \Rightarrow \text{ $q$-modified-log-Sobolev inequality } \Rightarrow \\
\text{$(\varphi_p,1)$ Transport-Entropy inequality } \Rightarrow \text{ $p$-exponential concentration inequality} ~.
\end{multline}
More precisely, there exist constants $C_1,C_2,C_3 > 0$ so that for any $p \in [1,2]$ (and setting $q = p^*$):
\[
D_{Iso_p} \leq C_1 D_{mLS_{q}} \leq C_2 D_{TE_{\varphi_p,1}} \leq C_3 D_{Con_p} ~.
\]

The third implication in (\ref{eq:hierarchy}) or (\ref{eq:hierarchy-2}) for general $p \geq 1$ is due to Marton \cite{Marton86,Marton96}. The fact that a $q$-modified-log-Sobolev inequality implies $p$-exponential concentration follows from the ``Herbst argument" (see \cite{Ledoux-Book,BobkovZegarlinski}).

The second implication in (\ref{eq:hierarchy}) or (\ref{eq:hierarchy-2}) for $p=q=2$ is due to Otto and Villani \cite{OttoVillaniHWI} in the manifold setting. This was subsequently given a different proof in the Euclidean setting using the Hamilton-Jacobi equation by Bobkov--Gentil--Ledoux \cite{BobkovGentilLedoux}, who also established the implication (along with its converse) for $p=1$. The implication was then extended to the range $p \in [1,2]$ and $p \geq 2$ in \cite{GentilGuillinMiclo}, \cite{BloghEtAl-qlogSob-pTE}, respectively. See also \cite[Theorem 22.28]{VillaniOldAndNew},\cite{WangTEInqsOnPathSpaces,LottVillaniHamiltonJacobi,Wang2008qLogSobAndTEInqsViaSuperPoincare,GozlanTEEquivDimFreeConc} for generalizations, extensions to more general measure-metric spaces, and further techniques.

The first implication in (\ref{eq:hierarchy}) or (\ref{eq:hierarchy-2}) for $p=q=2$ is due to M. Ledoux \cite{LedouxBusersTheorem}, later refined by Beckner (see \cite{LedouxLectureNotesOnDiffusion}).
In the general $p\geq 2$ case, it is due to Bobkov and Zegarlinski \cite{BobkovZegarlinski} (see also \cite{EMilmanSodinIsoperimetryForULC}).
For $p=1$, it follows from the characterization (\ref{eq:BGL-equiv}) and Cheeger's inequality \cite{CheegerInq,MazyaCheegersInq1}, which implies that $D_{Poin} \geq c D_{Iso_1}$ for some universal constant $c>0$. In the general $p \in [1,2]$ case, this implication is due to A. Kolesnikov \cite[Theorem 1.1]{KolesnikovModLogSobInqs}, but this requires some further explanation. Kolesnikov showed (in particular) that given $p \in (1,2]$, if $(\Real^n,\abs{\cdot},\mu)$ satisfies a $p$-exponential isoperimetric inequality (with constant $D_{Iso_p}$), then:
\begin{equation} \label{eq:Kol}
\exists D >0 \; \text{ s.t. } \; \forall f \in \log \F \;\;\;\;  Ent_\mu(f^2) \leq D \int f^2 \varphi_{*,q}\brac{C \frac{|\nabla f|}{|f|}} d\mu ~,
\end{equation}
with $C=1$, where the constant $D$ depends on $D_{Iso_p}$, $p$ and a lower bound on a variant of the Poincar\'e constant $D_{Poin'}$ (which is equivalent to it up to constants, see e.g. \cite[Lemma 2.1]{EMilman-RoleOfConvexity}). By Cheeger's inequality as above, we have $D_{Poin'} \geq c D_{Iso_1} \geq c' D_{Iso_p}$ for some universal constants $c,c'>0$, thus removing $D_{Poin'}$ from the list of parameters. A further careful inspection of the proof reveals that the estimate on $D$ does not depend on any integrability properties of the measure $\mu$, and hence dimension independent. The most delicate part is to notice that the estimate is actually uniform in $p \in (1,2]$, and hence extends to the case $p=1$, if one is willing to use a different universal constant $C>1$ in (\ref{eq:Kol}); this may also be seen from the simpler tightening procedure suggested by Barthe and Kolesnikov \cite[Theorem 2.4]{BartheKolesnikov}, using $C=2$. It follows that if $D_{Iso_p} = 1$, then (\ref{eq:Kol}) holds for some universal $D,C>0$ uniformly in $p \in [1,2]$.
Using (\ref{eq:varphi*_q-prop}), it follows that if $D_{Iso_p} = 1$ then $D_{mLS_q} \geq (C \max(\sqrt{D},1))^{-1}$, for all $p \in [1,2]$. But since our constants scale linearly in the metric, it must follow that $D_{mLS_q} \geq (C \max(\sqrt{D},1))^{-1} D_{Iso_p}$, concluding our claim. We also note that Kolesnikov's method is not particular to Euclidean space, and extends to the Riemannian-manifold-with-density setting.

\subsection{Reversing the Hierarchy} \label{subsec:pre-reverse}

In general, it is known that it is not possible to reverse any of the implications in (\ref{eq:hierarchy}) and (\ref{eq:hierarchy-2}), at least not for general $p$ in the corresponding range. That the first implication cannot be reversed follows for instance from known criteria for Hardy-type inequalities on $(\Real,\abs{\cdot},\mu)$ \cite{MuckenhouptHardyInq,MazyaBook,BobkovGotzeLogSobolev,BobkovZegarlinski,BartheRobertoModLogSobOnR}.
As for the second implication, this was settled by Cattiaux and Guillin \cite{CattiauxGuillinT2InqWeakerThanLogSob} in the case $p=q=2$; the extension to the case $q \geq 2, p \in (1,2]$ may be obtained by combining the results of Barthe--Roberto \cite{BartheRobertoModLogSobOnR} and Gozlan \cite{GozlanTEInqsOnR}, which respectively characterize in that range $q$-modified-log-Sobolev inequalities and $(\varphi_p,1)$ TE inequalities on the real line. As already mentioned, in the case $p=1$, $\infty$-modified-log-Sobolev and $(\varphi_1,1)$ TE inequalities are actually known to be equivalent \cite{BobkovGentilLedoux}.
The third implication is certainly false in general. Indeed, note that any compactly supported measure always satisfies a concentration inequality, since $\K(r) = +\infty$ for all $r > \Delta$, where $\Delta<\infty$ denotes the diameter of the support. On the other hand, as remarked in \cite{DjelloutGuillinWu}, since $(\varphi_p,1)$ TE inequalities for $p \in [1,2]$ imply a Poincar\'e inequality (see Section \ref{sec:TE-equiv}), it follows that when the support is in addition disconnected, the third implication cannot be reversed; a similar argument works for $p > 2$.

Some partial reversal results have been obtained under some additional assumptions, typically involving convexity. Under our convexity assumptions, it was shown by Ledoux \cite{LedouxSpectralGapAndGeometry} (extending Buser \cite{BuserReverseCheeger}) that a Poincar\'e inequality implies back a $1$-Exponential (or Cheeger type) isoperimetric inequality, up to universal constants. The semi-group method developed by Ledoux and Bakry--Ledoux \cite{BakryLedoux} allowed reversing general functional inequalities with a $\norm{|\nabla f|}_{L_2(\mu)}$ term under our semi-convexity assumptions, and in particular applies to the log-Sobolev inequality. This method was extended to handle general $\norm{|\nabla f|}_{L_q(\mu)}$ terms in \cite{EMilmanRoleOfConvexityInFunctionalInqs, EMilmanMazya}, and in particular applies to $q$-log-Sobolev inequalities. Under our semi-convexity assumptions, it was shown by Otto--Villani \cite{OttoVillaniHWI} (see also \cite{BobkovGentilLedoux} for a semi-group proof) that a strong-enough $(2,2)$ TE inequality implies back a log-Sobolev inequality, with dimension independent estimates. Under our convexity assumptions, this was extended to the general $p \in [1,2]$ case by Gentil--Guillin--Miclo \cite{GentilGuillinMiclo}; these authors comment that it would be possible to extend their result to handle the $\kappa$-semi-convexity assumptions ($\kappa > 0$), but in view of the spirit of our results, we are not certain this is so. This was also extended to the $p > 2$ case by Wang \cite{Wang2008qLogSobAndTEInqsViaSuperPoincare} in the Riemannian setting (but with dimension dependent bounds), and by Balogh et al. \cite{BloghEtAl-qlogSob-pTE} in a more general one (at least in the $\kappa=0$ case, but their proof should generalize to arbitrary $\kappa > 0$).

We do not know whether previous attempts have been considered to deduce isoperimetric, functional or TE inequalities from concentration inequalities. A weaker variant has been considered by many authors, including Wang \cite{WangIntegrabilityForLogSob,WangImprovedIntegrabilityForLogSob}, Chen--Wang \cite{ChenWangOptimalConstantInSubGaussianImpliesLogSob}, Bobkov \cite{BobkovGaussianIsoLogSobEquivalent}, Barthe \cite{BartheIntegrabilityImpliesIsoperimetryLikeBobkov}, Barthe--Kolesnikov \cite{BartheKolesnikov}, where an integrability condition of the measure, together with $\kappa$-semi-convexity assumptions, was shown to guarantee an appropriate isoperimetric or functional inequality. Unfortunately, as explained in \cite{EMilmanGeometricApproachPartI}, these types of results will unavoidably always yield dimension-dependent estimates.

However, it was shown in our previous work \cite{EMilmanGeometricApproachPartI} using tools from Riemannian Geometry (see also Ledoux \cite{LedouxConcentrationToIsoperimetryUsingSemiGroups} for an alternative approach), that under our semi-convexity assumptions, general concentration inequalities imply back their isoperimetric counterparts, with quantitative estimates which \emph{do not depend} on the dimension of the underlying manifold. This implies in particular that all of the tiers in our hierarchies are equivalent (up to universal constants) in this case. The precise formulation is as follows:

\begin{thm}[\cite{EMilmanGeometricApproachPartI}] \label{thm:main-equiv}
Let $\kappa \geq 0$ and let $\alpha: \Real_+ \rightarrow \Real \cup \set{+\infty}$ denote an increasing continuous function so that:
\begin{equation} \label{eq:alpha-cond}
\exists \delta_0 > 1/2 \;\;\; \exists r_0 \geq 0 \;\;\; \forall r \geq r_0 \;\;\; \alpha(r) \geq \delta_0 \kappa r^2 ~.
\end{equation}
Then under our $\kappa$-semi-convexity assumptions, the concentration inequality:
\[
\K(r) \geq \alpha(r) \;\;\; \forall r \geq 0
\]
implies the following isoperimetric inequality:
\begin{equation} \label{eq:i-2}
\tilde{\I}(v) \geq \min(c_{\delta_0} \; v \gamma(\log 1/v),c_{\kappa,\alpha}) \;\;\; \forall v \in [0,1/2] \;\;\; , \;\;\; \textrm{where} \;\; \gamma(x) = \frac{x}{\alpha^{-1}(x)} ~,
\end{equation}
and $c_{\delta_0},c_{\kappa,\alpha}>0$ are constants depending solely on their arguments. Moreover, if $\kappa = 0$, we may take $c_{\delta_0} = c$ and $c_{0,\alpha} = \frac{c}{4} \gamma(\log 4)$ for some universal constant $c>0$. If $\kappa > 0$, the dependence of $c_{\kappa,\alpha}$ on $\alpha$ may be expressed only via $\delta_0$ and $\alpha(r_0)$.
\end{thm}

In this work, we focus on some of the consequences of Theorem \ref{thm:main-equiv} to the study of Isoperimetric, Functional and Transport-Entropy inequalities. The three main motives appearing in this work are:
\begin{itemize}
\item Under our convexity assumptions, both hierarchies (\ref{eq:hierarchy}) and (\ref{eq:hierarchy-2}) may be reversed.
\item Under our $\kappa$-semi-convexity assumptions, the hierarchy (\ref{eq:hierarchy}) for $p>2$ may be reversed.
\item Under our $\kappa$-semi-convexity assumptions, the hierarchy (\ref{eq:hierarchy}) for $p=2$ may be reversed, if a strong-enough
concentration inequality is satisfied.
\end{itemize}
By going up and down the hierarchies, we deduce the various announced results.

\section{Stability with respect to $\snorm{\frac{d \mu_2}{d \mu_1}}_{L^\infty}$} \label{sec:stab1}

We start by deducing several stability results with respect to a rather restrictive notion of proximity of $\mu_2$ to $\mu_1$, given by $\snorm{\frac{d \mu_2}{d \mu_1}}_{L^\infty}$ (assuming that $\mu_2 \ll \mu_1$).

\subsection{Concentration Inequalities}

The main observation behind the contents of this section is the following elementary:

\begin{lem} \label{lem:conc-going-down}
Let $\mu_1$,$\mu_2$ denote two probability measures on a common metric space $(\Omega,d)$, and assume that:
\begin{equation} \label{eq:going-down}
\norm{\frac{d \mu_2}{d \mu_1}}_{L^\infty} \leq \exp(D) ~.
\end{equation}
Let $\K_i = \K_{(\Omega,d,\mu_i)}$ denote the corresponding log-concentration profiles,
and assume that $\K_1 \geq \alpha_1$ where $\alpha_1 : \Real_+ \rightarrow \Real_+$ is strictly increasing and continuous. Then:
\[
 \K_2(r + r_1) \geq \alpha_1(r) - D \;\;\; \forall r > 0 ~,~ \text{where} ~ r_1 := \alpha_1^{-1}(\log 2 + D) ~.
\]
\end{lem}
\begin{proof}
Recall that $\K_1 \geq \alpha_1$ just means that:
\begin{equation} \label{eq:conc}
\mu_1(B) \geq 1/2 \;\; \Rightarrow \;\; 1 - \mu_1(B_r) \leq \exp(-\alpha_1(r)) ~,
\end{equation}
which is easily seen to be equivalent to:
\begin{equation} \label{eq:conc-reversed}
\mu_1(A) > \exp(-\alpha_1(r)) \;\; \Rightarrow \;\; \mu_1(A_r) > 1/2 ~.
\end{equation}
Now let $A$ denote a Borel subset of $\Omega$ with $\mu_2(A) \geq 1/2$. The condition (\ref{eq:going-down}) implies that $\mu_1(A) \geq \exp(-(\log2 + D))$. It follows from (\ref{eq:conc-reversed}) and the definition of $r_1$ that $\mu_1(\overline{A_{r_1}}) \geq 1/2$. Applying (\ref{eq:going-down}) and (\ref{eq:conc}) once again, we deduce that:
\[
 \exp(-D) \mu_2(\Omega\setminus A_{r+r_1}) \leq \mu_1(\Omega \setminus A_{r+r_1}) \leq \exp(-\alpha_1(r)) ~,~ \forall r>0 ~.
\]
Recalling the definition of $\K_2$, the desired assertion follows.
\end{proof}

\subsection{Isoperimetric Inequalities}

\begin{thm} \label{thm:iso-stability}
Let $\mu_1$,$\mu_2$ denote two probability measures on a common Riemannian manifold $(M,g)$ so that $\mu_2 \ll \mu_1$, and assume that (\ref{eq:going-down}) holds.
 Assume that our $\kappa$-semi-convexity assumptions are satisfied for $(M,g,\mu_2)$ ($\kappa \geq 0$), and let $\gamma_1 : [\log 2, \infty) \rightarrow \Real_+$ denote a continuous positive function so that:
\begin{equation} \label{eq:gamma-growth}
 \exists \delta_0 > 1/2 \;\;\; \exists x_0 \geq \log 2 \;\;\; \forall x \geq x_0 \;\;\; \gamma_1(x) \geq 2 \sqrt{\delta_0 \kappa x} ~.
\end{equation}
Let $\I_i = \I_{(M,g,\mu_i)}$ denote the corresponding isoperimetric profiles. If $(M,g,\mu_1)$ satisfies the isoperimetric inequality:
\[
 \tilde{\I}_1(v) \geq v \gamma_1(\log 1/v) \;\;\; \forall v \in [0,1/2] ~,
\]
then $(M,g,\mu_2)$ satisfies the following isoperimetric inequality:
\[
 \tilde{\I}_2(v) \geq \min(c_{\delta_0} v \gamma_2(\log 1/v), c_{\kappa,\gamma_1,D}) \;\;\; \forall v \in [0,1/2] ~,
\]
where $\gamma_2 : [\log 2, \infty) \rightarrow \Real_+$ is defined as:
\[
\gamma_2(x) := \frac{x}{\int_{\log 2}^{x + D} \frac{dy}{\gamma_1(y)} } ~,
\]
and $c_{\delta_0}, c_{\kappa,\gamma_1,D} > 0$ depend solely on their arguments.
\end{thm}

\begin{rem}
It is obvious that the assertion of the theorem is completely false without the semi-convexity assumption on $(M,g,\mu_2)$ (e.g. consider
the case that the support of $\mu_2$ is disconnected).
\end{rem}

\begin{proof}
Using the same notations as in Lemma \ref{lem:conc-going-down}, if follows from (\ref{eq:i-c}) that:
\[
 \K_1(r) \geq \alpha_1(r) \;\;\; \forall r \geq 0 \;\;\; \textrm{where} \;\; \alpha_1^{-1}(x) = \int_{\log 2}^x \frac{dy}{\gamma_1(y)} ~.
\]
The growth condition (\ref{eq:gamma-growth}) on $\gamma_1$ ensures that $\alpha_1$ satisfies the following growth condition:
\begin{equation} \label{eq:alpha_1-growth}
 \exists \delta'_0 := \frac{1}{2}(\delta_0 + 1/2) > \frac{1}{2} \;\;\; \exists r'_0 = r'_0(\delta_0,\kappa,x_0) \;\;\; \forall r \geq r'_0 \;\;\; \alpha_1(r) \geq \delta'_0 \kappa r^2 ~.
\end{equation}

Applying Lemma \ref{lem:conc-going-down}, we deduce from (\ref{eq:going-down}) that:
\[
 \K_2(r) \geq \alpha_2(r) := \begin{cases} \alpha_1(r - r_1) - D &  r>2 r_1 \\ \log 2 & r \leq 2 r_1 \end{cases} ~,
\]
where:
\[
 r_1 := \alpha_1^{-1}(\log 2 + D) = \int_{\log 2}^{\log 2 + D} \frac{dy}{\gamma_1(y)} ~.
\]
Clearly $\alpha_2$ inherits the growth condition (\ref{eq:alpha_1-growth}) from $\alpha_1$:
\[
 \exists \delta''_0 := \frac{1}{2}(\delta'_0 + 1/2) > \frac{1}{2} \;\;\; \exists r''_0 = r''_0(r'_0,\delta'_0,D,\kappa,r_1) \;\;\; \forall r \geq r''_0 \;\;\; \alpha_2(r) \geq \delta''_0 \kappa r^2 ~.
\]
Since our $\kappa$-semi-convexity assumptions are satisfied for $(M,g,\mu_2)$ and since $\delta''_0 > 1/2$, we may apply Theorem \ref{thm:main-equiv} and deduce the isoperimetric inequality:
\[
 \tilde{\I}_2(v) \geq \min(c_{\delta''_0} v \gamma'_2(\log 1/v) , c_{\kappa,\alpha_2} ) \;\;\; \forall v \in [0,1/2] ~,
\]
where:
\[
 \gamma'_2(x) = \frac{x}{\alpha_2^{-1}(x)} = \frac{x}{\alpha_1^{-1}(x+D) + r_1} ~.
\]
Since $\gamma'_2(x) \geq \gamma_2(x)/2$ for $x \geq \log 2$, the proof is complete.
\end{proof}

\begin{cor} \label{cor:iso-stability}
Under measure perturbations of the form $\snorm{\frac{d\mu_2}{d\mu_1}}_{L^\infty} \leq \exp(D)$ and our $\kappa$-semi-convexity assumptions on $(M,g,\mu_2)$ ($\kappa \geq 0$):
\begin{enumerate}
\item
If $\kappa=0$, then $p$-exponential isoperimetric inequalities for $p \in [1,\infty)$ are stable under perturbation:
\[
D_{Iso_p}(M,g,\mu_2) \geq c^1_{p,D} D_{Iso_p}(M,g,\mu_1) ~.
\]
\item
If $\kappa > 0$, then $p$-exponential isoperimetric inequalities for $p \in (2,\infty)$ are stable under perturbation:
\[
D_{Iso_p}(M,g,\mu_2) \geq c^2(D_{Iso_p}(M,g,\mu_1),p,\kappa,D)  ~.
\]
\item
If $\kappa > 0$, then a strong-enough Gaussian isoperimetric inequality ($p=2$ case) is also stable under perturbation:
\[
D_{Iso_2}(M,g,\mu_1) > \sqrt{\kappa} \;\; \Rightarrow \;\; D_{Iso_2}(M,g,\mu_2) \geq c^3(D_{Iso_2}(M,g,\mu_1),\kappa,D) ~.
\]
\end{enumerate}
Here $c^1_{p,D}>0$ is a constant depending solely on its arguments, and $c^2(\Delta,p,\kappa,D),c^3(\Delta,\kappa,D)$ are functions depending solely on their arguments, which in addition are strictly positive as soon as $\Delta$ is.
\end{cor}
\begin{proof}
Recalling Remark \ref{rem:iso_p}, by definition of a $p$-exponential isoperimetric inequality we have:
\[
\tilde{\I}_{M,g,\mu_1}(v) \geq c D_{Iso_p}(M,g,\mu_1) v \log^{1/q} 1/v \;\;\; \forall v \in [0,1/2] ~,
\]
for some universal constant $c>0$ and $q = p^*$. The first two claims then easily follow from Theorem \ref{thm:iso-stability} applied to $\gamma_1(x) = c D_{Iso_p}(M,g,\mu_1) x^{1/q}$, after noting that:
\begin{equation} \label{eq:good-estimate}
\gamma_2(x)  = \frac{1}{p} \frac{c D_{Iso_p}(M,g,\mu_1) x}{(x+D)^{1/p} - (\log2)^{1/p}} \geq \frac{c D_{Iso_p}(M,g,\mu_1)}{p} \brac{\frac{\log 2}{\log 2 + D}}^{1/p} x^{1/q} \;\;\; \forall x \geq \log 2 ~.
\end{equation}
Note that in the first case ($\kappa=0$), the dependence on $D_{Iso_p}(M,g,\mu_1)$ may be shown to be linear, due to the remarks at the end of the formulation of Theorem \ref{thm:main-equiv}. The third claim follows similarly, once it is checked (see e.g. \cite{BobkovGaussianIsopInqViaCube}) that when $p=2$:
\[
\lim_{v \rightarrow 0+} \frac{\I_{(\Real,|\cdot|,\Gamma_2)}(v)}{v \sqrt{\log 1/v}} = \sqrt{2} ~,
\]
where $\Gamma_2$ denotes the standard Gaussian measure on $\Real$. This implies that if $D_{Iso_2}(M,g,\mu_1) > \sqrt{\kappa}$, then setting $\delta_0 := D_{Iso_2}(M,g,\mu_1)^2 / (2\kappa) > 1/2$, the condition (\ref{eq:gamma-growth}) is satisfied for some big enough $x_0$. This completes the proof.
\end{proof}

\begin{rem}
The case $\kappa=0$ and $p=1$ of Corollary \ref{cor:iso-stability} was also deduced in our previous work \cite{EMilman-RoleOfConvexity}. In that work, it was shown that one may use: \begin{equation} \label{eq:c_1,D}
c^1_{1,D} \simeq \frac{1}{1+D} ~,
\end{equation}
and that up to universal constants, this result is sharp. Indeed, (\ref{eq:c_1,D}) may also be seen from (\ref{eq:good-estimate}) in the proof of Theorem \ref{thm:iso-stability}. In fact, this can be extended to the following estimate:
\[
c^1_{p,D} \simeq_p \frac{1}{1+D^{1/p}} ~.
\]
\end{rem}

\begin{rem}
It is also possible to derive an analogue of Lemma \ref{lem:conc-going-down} for the case when the roles of $\mu_1,\mu_2$ are interchanged, so that condition
(\ref{eq:going-down}) is replaced by $\snorm{\frac{d \mu_1}{d \mu_2}}_{L^\infty} \leq D$,  when $D \in (1,2)$. Unfortunately, in this case, we can only deduce a lower bound on $\K_2$ which will always be
smaller than $\log \frac{D}{D-1}$, and in particular, we cannot deduce that $\K_2$ increases to infinity. Repeating the arguments of Theorem \ref{thm:iso-stability}, this would only allow us to deduce an isoperimetric inequality of the form $\tilde{\I}_2(v) \geq c_{\kappa,\gamma_1,D}$ in the range $v \in [\lambda,1/2]$, for some $\lambda<1/2$ sufficiently close to $1/2$. This is not very useful in general, except in the case that our convexity assumptions are satisfied ($\kappa=0$), where this is already enough to imply a \emph{linear} isoperimetric inequality - see \cite{EMilman-RoleOfConvexity} and Subsection \ref{subsec:stab2-W1}.
\end{rem}

\subsection{Log-Sobolev Inequalities}

Although it is possible to formulate an analogue to Theorem \ref{thm:iso-stability} in the language of more general functional and Transport-Entropy inequalities, we prefer to restrict ourselves in this subsection to log-Sobolev inequalities, since these lie on the border of our method and are the most interesting in applications.

Before stating our result, let us first recall the following
well-known stability result due to Holley and Stroock \cite{HolleyStroockPerturbationLemma}:

\begin{lem}[Holley--Stroock]
Let $\mu_i = \exp(-V_i(x)) dvol_M(x)$ ($i=1,2$) denote two probability measures on a common Riemannian manifold $(M,g)$,
so that:
\begin{equation} \label{eq:HS-two-sided}
 V_1 + D_+ \geq V_2 \geq V_1 - D_{-} ~.
\end{equation}
If $(M,g,\mu_1)$ satisfies a log-Sobolev inequality:
\[
 \exists \rho > 0 \;\;\; \rho Ent_{\mu_1}(f^2) \leq \int | \nabla f |^2 d\mu_1 \;\;\; \forall f \in \F ~,
\]
then $(M,g,\mu_2)$ also satisfies a log-Sobolev inequality:
\[
 \rho \exp(-(D_{+} + D_{-}))  Ent_{\mu_2}(f^2) \leq \int | \nabla f |^2 d\mu_2 \;\;\; \forall f \in \F ~.
\]
\end{lem}

This result was obtained in the context of Statistical Mechanics, where $\set{V_i}$ represent some Hamiltonian potentials, and the condition (\ref{eq:HS-two-sided}) is interpreted as the assumption that $V_2$ is a bounded perturbation of $V_1$.
Although it is quite useful in various situations in this context, the condition (\ref{eq:HS-two-sided}) is unavoidably restrictive, since it is easy to check that there can be no stability in general without assuming both the upper and lower bounds on the perturbation. The following result, on the other hand, permits to dispose of the upper bound in (\ref{eq:HS-two-sided}), at the expense of an additional semi-convexity assumption; in addition, under our convexity assumptions, the quantitative dependence on the perturbation parameter is improved from exponential to linear:

\begin{thm} \label{thm:log-Sob-stability}
Let $\mu_i = \exp(-V_i(x)) dvol_M(x)$ ($i=1,2$) denote two probability measures on a common Riemannian manifold $(M,g)$. Assume that:
\[
V_2 \geq V_1 - D_{-} \;\;\; \text{and} \;\;\; \text{$(M,g,\mu_2)$ satisfies our $\kappa$-semi-convexity assumptions} ~,
\]
with some $\kappa \geq 0$. If $(M,g,\mu_1)$ satisfies a strong-enough log-Sobolev inequality:
\begin{equation} \label{eq:ls-assump}
 \exists \rho > \kappa/2 \;\;\text{ such that }\;\; \rho Ent_{\mu_1}(f^2) \leq \int | \nabla f |^2 d\mu_1 \;\;\; \forall f \in \F ~,
\end{equation}
then $(M,g,\mu_2)$ satisfies a log-Sobolev inequality:
\begin{equation} \label{eq:ls-concl}
 C_{\rho,\kappa,D_{-}} Ent_{\mu_2}(f^2) \leq \int | \nabla f |^2 d\mu_2 \;\;\; \forall f \in \F ~,
\end{equation}
where $C_{\rho,\kappa,D_{-}} > 0$ depends solely on its arguments. Moreover, when $\kappa = 0$, one may use:
\[
C_{\rho,0,D_{-}} = \rho \frac{c}{1+D_{-}} ~,
\]
where $c>0$ is a universal constant.
\end{thm}

\begin{proof}
As usual, we use the same notations as in Lemma \ref{lem:conc-going-down}.
By the Herbst argument (see \cite{Ledoux-Book}), which was already mentioned in Section \ref{sec:pre}, it is known that the log-Sobolev inequality (\ref{eq:ls-assump}) implies the following Laplace-functional inequality:
\[
 \int \exp(\lambda f) d\mu_1 \leq \exp(\lambda^2/(4\rho)) \;\;\; \forall \lambda \geq 0 \;\; \forall \text{ $1$-Lipschitz $f$ s.t. } \int f d\mu_1 = 0 ~.
\]
It is easy to check (see e.g. Lemma \ref{lem:conc-Laplace}) that this implies the following concentration inequality on $(M,g,\mu_1)$:
\[
\K_1(r) \geq \rho \brac{r - \sqrt{\frac{\log2}{\rho}}}_+^2 \;\;\; \forall r \geq 0 ~.
\]
Using Lemma \ref{lem:conc-going-down} we deduce that:
\[
\K_2(r) \geq \alpha_2(r) := \rho \brac{r - 2 \sqrt{\frac{\log 2}{\rho}} - \sqrt{\frac{\log2 + D_{-}}{\rho}}}_+^2 - D_{-} ~.
\]
Since $\delta_0 := \frac{\rho}{\kappa} > 1/2$, $\K_2$ clearly satisfies the growth condition required to apply Theorem \ref{thm:main-equiv}:
\[
 \exists \delta'_0 := \frac{1}{2}(\delta_0 + 1/2) > \frac{1}{2} \;\;\; \exists r'_0 = r'_0(\rho,\kappa,D_{-}) \;\;\; \forall r \geq r'_0 \;\;\; \alpha_2(r) \geq \delta'_0 \kappa r^2 ~.
\]
Consequently, Theorem \ref{thm:main-equiv} implies that the following isoperimetric inequality is satisfied:
\begin{eqnarray*}
\tilde{\I}_2(v) & \geq & \min\brac{c_{\delta_0'} \sqrt{\rho} v \frac{\log 1/v}{\sqrt{\log 1/v + D_{-}} + 2 \sqrt{\log 2} + \sqrt{\log 2 + D_{-}}} , c_{\rho,\kappa,D_{-}}} \\
& \geq & c'_{\rho,\kappa,D_{-}} v \sqrt{\log 1/v} \;\;\;\;\;\; \forall v \in [0,1/2] ~.
\end{eqnarray*}
This means that $(M,g,\mu_2)$ satisfies a Gaussian (or 2-exponential) isoperimetric inequality in the notation of Section \ref{sec:pre}. As described there, it is known that this implies the log-Sobolev inequality (\ref{eq:ls-concl}), concluding the proof.
Note that when $\kappa = 0$, the remarks at the end of the formulation of Theorem \ref{thm:main-equiv} imply that one may use $c'_{\rho,\kappa,D_{-}} = c \sqrt{\rho/(1+D_{-})}$ above, with $c>0$ a universal constant, implying the last assertion of the theorem.
\end{proof}

A natural situation where only the lower bound in (\ref{eq:HS-two-sided}) is available, is when $\mu_2$ is the restriction of $\mu_1$ onto some ``event'' having positive probability. We state this explicitly as an immediate corollary of Theorem \ref{thm:log-Sob-stability}. To further elucidate this scenario, we restrict ourselves to the case $\kappa=0$ in the Euclidean setting, although it is of course possible to formulate the following more generally:

\begin{cor} \label{cor:log-Sob}
Let $\mu_1 = \exp(-V(x)) dx$ denote a probability measure on $\Real^n$. Let $A \subset \Real^n$ be such that:
\[
 \mu_1(A) = p > 0 \;\;\; \text{and} \;\;\; \text{$A$ is convex and $V$ is convex on $A$} ~.
\]
Set $\mu_2 = \mu_1|_A / \mu_1(A)$. If $(\Real^n,\abs{\cdot},\mu_1)$ satisfies a log-Sobolev inequality:
\[
 \exists \rho > 0 \;\;\; \rho Ent_{\mu_1}(f^2) \leq \int | \nabla f |^2 d\mu_1 \;\;\; \forall f \in \F ~,
\]
then $(\Real^n,\abs{\cdot},\mu_2)$ satisfies a log-Sobolev inequality:
\[
 c \frac{\rho}{1 + \log 1/p} Ent_{\mu_2}(f^2) \leq \int | \nabla f |^2 d\mu_2 \;\;\; \forall f \in \F ~,
\]
where $c > 0$ is a universal constant.
\end{cor}

\begin{rem}
Analogous stability results may be obtained for single-sided perturbations of other functional or TE inequalities, under
our semi-convexity assumptions. We mention here an analogue of the Holley--Stroock two-sided perturbation lemma, only
recently obtained by Gozlan, Roberto and Samson \cite{GozlanRobertoSamsonTEStability}, for Talagrand's $(2,2)$ Transport-Entropy inequality (and more general ones), which
yields the same exponential dependence on $D_+$ and $D_-$. As for the log-Sobolev inequality, this may be substantially improved under our semi-convexity assumptions.
\end{rem}

\section{Interlude: Concentration via Transport-Entropy Inequalities} \label{sec:inter}

In this section, we set the ground for the next sections, which deal with Transport-Entropy inequalities. Besides recalling known results, we show a complete equivalence between concentration and certain Transport-Entropy inequalities, which may be of independent interest.

\subsection{Weak Transport-Entropy inequalities}

\begin{dfn*}
We will say that $(\Omega,d,\mu)$ satisfies a (weak) Laplace-functional inequality if there exists $D > 0$, $\eps,\delta \geq 0$ and an increasing convex function
$\Phi : \Real_+ \rightarrow \Real_+$ with $\Phi(0) = 0$, so that:
\begin{eqnarray} \label{eq:trunc-Laplace}
 \int \exp(\lambda D f) d\mu \leq \exp\brac{\lambda \eps + \Phi^*(\lambda) + \delta} \;\;\; \forall \lambda \geq 0 \;\; \forall \text{ $1$-Lipschitz $f$ s.t. } \int f d\mu = 0 ~,
\end{eqnarray}
where $\Phi^*(\lambda) := \sup_{x \geq 0} \lambda x - \Phi(x)$ denotes the Legendre transform of $\Phi$.
\end{dfn*}

Recall that by the Monge-Kantorovich-Rubinstein dual characterization of $W_1$ (e.g. \cite[5.16]{VillaniOldAndNew}), we have that:
\begin{equation} \label{eq:MKR-duality}
W_1(\nu,\mu) = \sup \set{ \int f d\nu - \int f d\mu \; ; \; f \text{ is a 1-Lipschitz function on $(\Omega,d)$ } } ~.
\end{equation}
This characterization is the key ingredient in the
following (mild adaptation of a) theorem of Bobkov and G\"{o}tze \cite{BobkovGotzeLogSobolev} (see also \cite{VillaniOldAndNew}):

\begin{thm}[Bobkov--G\"{o}tze] \label{thm:BG}
Let $\Phi : \Real_+ \rightarrow \Real_+$ denote an increasing convex function so that $\Phi(0) = 0$. Then for any $\delta,\eps,D \geq 0$, the following weak Transport-Entropy inequality:
\begin{eqnarray} \label{eq:lem-trunc-TC}
 D W_1(\nu,\mu) \leq \Phi^{-1}(H(\nu | \mu) + \delta) + \eps \;\;\; \forall \text{ probability measure } \nu \ll \mu ~;
\end{eqnarray}
is equivalent to the weak Laplace-functional inequality (\ref{eq:trunc-Laplace}).
\end{thm}
\begin{proof}[Sketch of Proof]
By the dual characterization (\ref{eq:MKR-duality}) of $W_1$, the definitions of $\Phi^*$ and $H(\nu | \mu)$, and denoting $\theta = \frac{d\nu}{d\mu}$, (\ref{eq:lem-trunc-TC}) is equivalent to the statement that:
\[
 D \brac{\int f \theta d\mu - \int f d\mu} \leq \inf_{\lambda \geq 0} \frac{\Phi^*(\lambda) + Ent_{\mu}(\theta) + \delta}{\lambda} + \eps  ~,
\]
for any $1$-Lipschitz function $f$ and non-negative $\mu$-integrable function $\theta$ so that $\int \theta d\mu = 1$. Denoting:
\[
 \psi := \lambda D f - \lambda D \int f d\mu - \Phi^*(\lambda) - \eps \lambda - \delta ~,
\]
we see that $\int \psi \theta d\mu \leq Ent_\mu(\theta)$, for all $\theta$ as above. This is well known to be equivalent to $\int \exp(\psi) d\mu \leq 1$, which is equivalent to (\ref{eq:trunc-Laplace}).
\end{proof}

This gives rise to the following:
\begin{dfn*}
We will say that $(\Omega,d,\mu)$ satisfies a weak $(1,p)$ Transport-Entropy inequality ($p \geq 1$) if:
\[
\exists D>0 \;\;\; D W_1(\nu,\mu) \leq H(\nu | \mu)^{1/p} + 1 \;\;\; \forall \text{ probability measure } \nu ~.
\]
The best constant $D$ above will be denoted by $D_{wTE_{1,p}} = D_{wTE_{1,p}}(\Omega,d,\mu)$.
\end{dfn*}

It is well known that the weak Laplace-functional inequality (\ref{eq:trunc-Laplace}) is equivalent to our usual notion of concentration inequality. This is made precise in the following:

\begin{lem} \label{lem:conc-Laplace}
Let $\Phi : \Real_+ \rightarrow \Real_+$ denote an increasing convex function so that $\Phi(0) = 0$.
\begin{enumerate}
\item
If $\eps,\delta,D \geq 0$ then (\ref{eq:trunc-Laplace}) implies:
\begin{equation} \label{eq:trunc-conc}
 \K_{(\Omega,d,\mu)}(r) \geq \Phi((D' r - z')_+) - \delta' \;\;\; \forall r \geq 0 ~,
\end{equation}
with $D' = D$, $\delta'=\delta$ and $z' = \Phi^{-1}(\log 2 + \delta) + 2 \eps$.
\item
If $z',D' \geq 0$ and $\delta' \geq -\log 2$ then for any $\tau \in (0,1)$, (\ref{eq:trunc-conc}) implies (\ref{eq:trunc-Laplace}) with $D = \tau D'$, $\delta = \delta' + \log (\exp(-\delta') + \frac{\tau}{1 -\tau})$ and $\eps = \tau \brac{2z' + \Phi^{-1}(\log 2 + \delta') + \int_0^\infty \exp(-\Phi(r)) dr}$.
\end{enumerate}
\end{lem}
\begin{proof}[Sketch of proof]
We will show statement (2), statement (1) is simpler and follows along the same lines. It is immediate that (\ref{eq:trunc-conc}) is equivalent to:
\begin{equation} \label{eq:med-bound}
\mu\set{ f \geq med_\mu f + r} \leq \exp(-\Phi((rD' - z')_+) + \delta') \;\;\; \forall r \geq 0 \;\;\; \forall \text{ 1-Lipschitz function $f$} ~,
\end{equation}
where $med_\mu f$ denotes a median of $f$ with respect to $\mu$, i.e. a value so that $\mu(f \geq med_\mu
f) \geq 1/2$ and $\mu(f \leq med_\mu f) \geq 1/2$. Now let $f$ denote a 1-Lipschitz function with $\int f d\mu = 0$. Using (\ref{eq:med-bound}) to evaluate:
\[
\abs{ \int f d\mu - med_\mu f} \leq \int |f - med_\mu f| d\mu \leq r_0' + \int_{r_0'}^\infty \mu(|f - med_\mu f|\geq r) dr ~,
\]
with $r_0':= (z' + \Phi^{-1}(\log 2 + \delta'))/D'$, one checks using (\ref{eq:med-bound}) again that:
\begin{equation} \label{eq:exp-bound}
\mu\set{ f \geq r} \leq \mu\set{ f \geq med_\mu f - \abs{ \int f d\mu - med_\mu f}  + r} \leq \exp(-\Phi((D'r - z_0')_+) + \delta') \;\;\; \forall r \geq 0 ~,
\end{equation}
where $z_0' := 2z' + \Phi^{-1}(\log 2 + \delta') + \int_0^\infty \exp(-\Phi(r)) dr$.
Integrating by parts again and using (\ref{eq:exp-bound}), we evaluate:
\begin{eqnarray*}
\int \exp(D' \lambda f) d\mu & \leq & \exp(\lambda z_0') + \int_{z_0'}^\infty \lambda \exp(\lambda s) \mu\set{ f \geq s/D'} ds  \\
& \leq & \exp(\lambda z_0') \brac{ 1+ \exp(\delta') \lambda  \int_0^\infty \exp(\lambda s - \Phi(s)) ds } ~.
\end{eqnarray*}
Using that $\lambda s - \Phi(s) \leq \Phi^*(\lambda/\tau) - \frac{1-\tau}{\tau} \lambda s$ for all $s \geq 0$, it follows that:
\[
\int \exp(D' \lambda f) d\mu \leq \exp(\lambda z_0') \brac{1 + \frac{\tau}{1-\tau} \exp(\delta' + \Phi^*(\lambda/\tau))} ~,
\]
from which point it is immediate to verify the claim.
\end{proof}

\begin{rem} \label{rem:conc-Laplace}
Note that since (\ref{eq:trunc-conc}) is only meaningful whenever the right-hand side exceeds $\log 2$,
we can always change the values of $D',z',\delta'$ to $D'',z'',\delta''$ so that $z'' \geq 0$ and $\delta'' \geq -\log 2$ are arbitrary, as long as one of these inequalities is strict. Since $H(\nu | \mu) \geq 0$, the same applies to (\ref{eq:lem-trunc-TC}) and changing the values of $D,\eps,\delta$ to $D_2,\eps_2,\delta_2$ so that $\eps_2 \geq 0$ and $\delta_2 \geq 0$ are arbitrary, as long as one of these inequalities is strict. In either case, $D'' > 0$ or $D_2 > 0$ are determined in a manner depending on all the other parameters and in addition on $\Phi^{-1}$.
\end{rem}

Combining Theorem \ref{thm:BG} with the equivalence between weak Laplace-functional and concentration inequalities given by Lemma \ref{lem:conc-Laplace} and Remark \ref{rem:conc-Laplace}, it is easy to check that weak $(1,p)$ Transport-Entropy and $p$-exponential concentration inequalities are precisely equivalent:

\begin{cor} \label{cor:conc-weak-TE}
$D_{Con_p} \simeq D_{wTE_{1,p}}$ uniformly in $p \geq 1$.
\end{cor}

\subsection{Tight Transport-Entropy inequalities}

We will also need to use a ``tight" form of our weak $(1,p)$ Transport-Entropy inequalities for some of the results in the next sections. This is summarized in the following proposition,
which extends beyond the previously mentioned results in this section,
and which may be of independent interest:

\begin{prop} \label{prop:con-TE}
The following inequalities are equivalent:
\begin{enumerate}
\item
The $p$-exponential concentration inequality:
\begin{equation} \label{eq:lem-con-p}
\K(r) \geq (D_{Con_p} r)^p - 1 \;\;\; \forall r \geq 0 ~.
\end{equation}

\item
The weak $(1,p)$ Transport-Entropy inequality:
\begin{equation} \label{eq:lem-wTE-p}
D_{wTE_{1,p}} W_1(\nu,\mu) \leq H(\nu | \mu)^{1/p} + 1 \;\;\;  \forall \text{ probability measure } \nu ~.
\end{equation}

\item
The $(1,\varphi_p)$ Transport-Entropy inequality:
\begin{equation} \label{eq:lem-TE-ql}
D_{TE_{1,\varphi_p}} W_1(\nu,\mu) \leq \varphi_{p}^{-1}(H(\nu | \mu)) \;\;\;  \forall \text{ probability measure } \nu ~,
\end{equation}
where, recall, $\varphi_p$ is given by (\ref{eq:def-varphi}).
\end{enumerate}
The equivalence is in the sense that the best constants above satisfy $D_{Con_p} \simeq D_{wTE_{1,p}} \simeq D_{TE_{1,\varphi_p}}$ uniformly in $p \geq 1$.
\end{prop}

\begin{proof}
By Corollary \ref{cor:conc-weak-TE}, $D_{Con_p} \simeq D_{wTE_{1,p}}$ uniformly in $p \geq 1$, so it remains to prove that $D_{wTE_{1,p}} \simeq D_{TE_{1,\varphi_p}}$ uniformly. It will be convenient to slightly change our normalization, so we remark that $D_{wTE'_{1,p}} := p^{1/p} D_{wTE_{1,p}}$ is clearly the best constant in the following inequality:
\begin{equation} \label{eq:lem-wTE'-p}
D_{wTE'_{1,p}} W_1(\nu,\mu) \leq (p H(\nu | \mu))^{1/p} + p^{1/p} \;\;\;  \forall \text{ probability measure } \nu ~.
\end{equation}
Since $\varphi_{p}^{-1}(x) \leq p^{1/p}(x^{1/p} + 1)$, it is immediate that $D_{wTE'_{1,p}} \geq D_{TE_{1,\varphi_p}}$. The other direction is the tricky part. We will assume that $p>1$, the case $p=1$ follows by approximation. By Theorem \ref{thm:BG}, (\ref{eq:lem-wTE'-p}) is equivalent to the statement that:
\begin{equation} \label{eq:prop-assumption}
\int \exp(D_{wTE'_{1,p}} \lambda f) d\mu \leq \exp\brac{\lambda p^{1/p} + \frac{\lambda^q}{q}} \;\;\; \forall \text{ $1$-Lipschitz $f$ s.t. } \int f d\mu = 0 ~.
\end{equation}
To conclude, we will need to deduce from this that for some universal constant $c>0$:
\begin{equation} \label{eq:prop-conclusion}
\int \exp(\lambda c D_{wTE'_{1,p}} f) d\mu \leq \exp\brac{\varphi_{*,q}(\lambda)} \;\;\; \forall \lambda \geq 0 \;\;\; \forall \text{ $1$-Lipschitz $f$ s.t. } \int f d\mu = 0 ~,
\end{equation}
which will imply that $D_{TE_{1,p}} \geq c D_{wTE'_{1,p}}$ by Theorem \ref{thm:BG} (recall that $\varphi_{*,q} = (\varphi_p)^*$ is given by (\ref{eq:def-varphi*})).
By taking $c>0$ smaller than some universal constant $c_0>0$, it is easy to check that (\ref{eq:prop-assumption}) implies (\ref{eq:prop-conclusion}) for $\lambda \geq 1$, so it remains to check (\ref{eq:prop-conclusion}) in the range $\lambda \in [0,1]$.

Fix any $1$-Lipschitz function $f$ so that $\int f d\mu = 0$. We proceed by denoting $\L(\lambda) := \log \int \exp(\lambda f) d\mu$, the logarithm of the Laplace transform. Note that $\L(\lambda) \geq 0$ by Jensen's inequality. Further denoting $\mu_\lambda := \exp(\lambda f) \mu / \int \exp(\lambda f) d\mu$, it is immediate to check that:
\[
\L''(\lambda) = \int f^2 d\mu_\lambda - (\int f d\mu_\lambda)^2 \leq \frac{\int f^2 \exp(\lambda f) d\mu}{\int \exp(\lambda f) d\mu} \leq \int f^2 \exp(\lambda f) d\mu ~.
\]
By the Cauchy-Schwarz inequality, it follows that:
\[
\L''(\lambda) \leq \brac{\int f^4 d\mu}^{\frac{1}{2}} \brac{\int \exp(2 \lambda f) d\mu}^{\frac{1}{2}} ~.
\]
Hence, using (\ref{eq:prop-assumption}) with $\lambda=1$ and a standard application of the Markov--Chebyshev inequality and integration by parts, we conclude that $\L''(\lambda) \leq C^2/{D_{wTE'_{1,p}}^2}$ whenever $\lambda \leq D_{wTE'_{1,p}}/2$, for some universal constant $C>0$. Since $\L(0) = 0$ and $\L'(0) = \int f d\mu = 0$, we conclude that $\L(\lambda) \leq \frac{1}{2} (C \lambda / D_{wTE'_{1,p}})^2$ for $\lambda \in [0,D_{wTE'_{1,p}}/2]$. Denoting $c := \min(1/C , 1/2, c_0)$, this implies that:
\[
\int \exp(c D_{wTE'_{1,p}} \lambda f) d\mu \leq \exp\brac{\lambda^2/2} \leq \exp\brac{\varphi_{*,q}(\lambda)}\;\;\; \forall \lambda \in [0,1] ~.
\]
This confirms the validity of (\ref{eq:prop-conclusion}) in the range $\lambda \in [0,1]$ and concludes the proof.
\end{proof}

\begin{rem}
A previous characterization of $(1,2)$ Transport-Entropy inequalities was obtained by Djellout, Guillin and Wu \cite{DjelloutGuillinWu}, and strengthened by Bolley and Villani \cite{BolleyVillani}. This was generalized to more general $(1,\phi)$ TE inequalities by Gozlan and Leonard \cite{GozlanLeonard}. All of these characterizations were in terms of an integrability condition of the form $B = \int \exp(\phi(d(x,x_0))) d\mu(x) < \infty$ for some (equivalently, all) $x_0 \in \Omega$. The problem with these criteria is that they all inevitably result in bad quantitative dependence when trying to estimate $D_{TE_{1,\phi}}$ via $B$ and $\phi$; in particular, when the underlying space is an $n$-dimensional manifold, they will all result in dimension dependent bounds (see \cite{EMilmanGeometricApproachPartI}). Proposition \ref{prop:con-TE} (which clearly extends to more general $(1,\phi)$ TE inequalities) demonstrates that the right characterization is via concentration inequalities (as opposed to integrability criteria). In addition, all of the above mentioned criteria may be easily recovered from it, since it is easy to check (see e.g. \cite[Section 7.2]{EMilmanGeometricApproachPartI}) that:
\[
\K(r) \geq \phi((r - \phi^{-1}(\log 2B))_+) - \log B \;\;\; \forall r \geq 0 ~.
\]
\end{rem}

\begin{rem} \label{rem:no-tight}
By checking what happens for measures with compact yet disconnected support, it is not difficult to realize that the equivalence between the weak and tight Transport-Entropy inequalities is rather special to the $W_1$ distance, and that analogous results cannot hold in general for $W_p$, $p > 1$.
However, under our semi-convexity assumptions, it is in fact possible to tighten these weak TE inequalities, by passing through the appropriate concentration inequality and employing Theorem \ref{thm:main-equiv}.
\end{rem}

\section{Stability under Wasserstein distance perturbation} \label{sec:stab2}

A drawback of using a distance of the form $\snorm{\frac{d\mu_2}{d\mu_1}}_{L^\infty}$ (or $\snorm{\frac{d\mu_1}{d\mu_2}}_{L^\infty}$) to measure the extent of a perturbation when analyzing the stability of various inequalities, as in Section \ref{sec:stab1}, is that the estimates become meaningless when the measures $\mu_1,\mu_2$ have disjoint supports, or more generally, are mutually singular.
In this section, we analyze the stability with respect to several new distances, which provide further flexibility and generalize some of the previous results:
\begin{itemize}
\item
The Wasserstein distance $W_1(\mu_1,\mu_2)$ and consequently the relative entropies $H(\mu_1 | \mu_2)$ and $H(\mu_2 | \mu_1)$.
\item
A new distance $\widetilde{W}_{\Psi_1}(\mu_1,\mu_2)$ which we introduce, called the $\Psi_1$-Lipschitz metric.
\end{itemize}

\subsection{Stability under $\widetilde{W}_{\Psi_1}$ perturbation}

Let $\M_{\Psi_1}$ denote the space of probability measures $\mu$ satisfying that $\int_\Omega \exp(\lambda d(x,x_0)) d\mu < \infty$ for any $\lambda > 0$ and some (any) $x_0 \in \Omega$. On this space, we introduce the following distance:
\begin{dfn*}
\begin{equation} \label{eq:def-Wpsi1}
\widetilde{W}_{\Psi_1}(\nu,\mu) := \sup \set{ \frac{\abs{\log \int \exp(g) d\nu - \log \int \exp(g) d\mu}}{\norm{g}_{Lip}} \; ; \; g  \text{ is a Lipschitz function on $(\Omega,d)$ } } ~.
\end{equation}
\end{dfn*}

It is clear that $\widetilde{W}_{\Psi_1}$ satisfies the triangle inequality and that it is symmetric. It is also easy to see that $\widetilde{W}_{\Psi_1}(\nu,\mu) = 0$ if and only if $\nu = \mu$, for instance by using the Hahn-Banach theorem together with the Stone-Weierstrass theorem (in its lattice version) and the fact that functions of the form $\exp(g)$ as above separate points in a metric space. We consequently verify that $\widetilde{W}_{\Psi_1}$ is a metric on $\M_{\Psi_1}$, which we call the $\Psi_1$-Lipschitz metric.
Another immediate property using the Monge-Kantorovich-Rubinstein dual characterization (\ref{eq:MKR-duality}) of $W_1$, is that:
\begin{equation} \label{eq:1-Psi_1}
W_1(\nu,\mu) \leq \widetilde{W}_{\Psi_1}(\nu,\mu) ~.
\end{equation}
Indeed, this is seen by testing in (\ref{eq:def-Wpsi1}) functions $g$ of the form $\eps f$, where $f$ is an arbitrary 1-Lipschitz function, and taking the limit as $\eps$ tends to $0$.

We comment that, in analogy to the dual characterization (\ref{eq:MKR-duality}) of $W_1$,
this metric may have some relation to the more standard $\Psi_1$-Wasserstein metric:
\[
W_{\Psi_1}(\nu,\mu) := \inf \set{\lambda > 0 \; ; \; \inf_{\pi \in \M(\nu,\mu)} \int_{\Omega \times \Omega} \exp(c(x,y)/\lambda) d\pi(x,y) \leq 2 } ~,
\]
but we have not been able to make this relation precise.

\medskip

The key reason to use the $\widetilde{W}_{\Psi_1}$ metric is the following lemma, which asserts that (weak) Laplace-functional inequalities are stable under perturbations in this metric:
\begin{lem}
Let $\mu_1,\mu_2$ denote two Borel probability measures on a common metric space $(\Omega,d)$. Assume that $(\Omega,d,\mu_1)$ satisfies the following (weak) Laplace-functional inequality:
\[
\int \exp(\lambda f) d\mu_1 \leq \exp\brac{\Phi(\lambda)} \;\;\; \forall \lambda \geq 0 \;\; \forall \text{ $1$-Lipschitz $f$ s.t. } \int f d\mu_1 = 0 ~,
\]
where $\Phi : \Real_+ \rightarrow \Real_+$ is an arbitrary function. Then $(\Omega,d,\mu_2)$ satisfies the following weak Laplace-functional inequality:
\[
\int \exp(\lambda f) d\mu_2 \leq \exp\brac{\Phi(\lambda) + 2 \lambda \widetilde{W}_{\Psi_1(\mu_1,\mu_2)}} \;\;\; \forall \lambda \geq 0 \;\; \forall \text{ $1$-Lipschitz $f$ s.t. } \int f d\mu_2 = 0 ~.
\]
\end{lem}
\begin{proof}
By definition, we have that:
\[
\int \exp(\lambda f) d\mu_2 \leq \int \exp(\lambda f) d\mu_1 \exp(\lambda \widetilde{W}_{\Psi_1(\mu_1,\mu_2)}) \;\;\; \forall \lambda \geq 0 \;\; \forall \text{ $1$-Lipschitz $f$} ~.
\]
If $\int f d\mu_2 = 0$ for a $1$-Lipschitz function $f$, we know by (\ref{eq:MKR-duality}) and (\ref{eq:1-Psi_1}) that $\int f d\mu_1 \leq W_1(\mu_1,\mu_2) \leq \widetilde{W}_{\Psi_1}(\mu_1,\mu_2)$, so we can bound the first term on the right hand side above by:
\[
\int \exp\brac{\lambda (f - \int f d\mu_1)} d\mu_1 \; \exp\brac{\lambda \int f d\mu_1} \leq \exp\brac{\Phi(\lambda) + \lambda \widetilde{W}_{\Psi_1}(\mu_1,\mu_2)} ~.
\]
This completes the proof.
\end{proof}

Since weak Laplace-functional inequalities are equivalent to concentration inequalities by Lemma \ref{lem:conc-Laplace}, we deduce that concentration inequalities are also stable under $\widetilde{W}_{\Psi_1}$ perturbation. Hence, as in the previous section, this stability may be transferred to the level of isoperimetric and functional inequalities, once our semi-convexity assumptions on $(\Omega,d,\mu_2)$ are satisfied. The formulations are completely analogous to those of Theorem \ref{thm:iso-stability}, Corollary \ref{cor:iso-stability} and Theorem \ref{thm:log-Sob-stability}, so we do not explicitly state them here, and leave this to the interested reader.

\subsection{Stability of Linear (Cheeger) Isoperimetric Inequalities under Convexity Assumptions} \label{subsec:stab2-W1}

For the results of this subsection, let us recall some further notation and results from \cite{EMilman-RoleOfConvexityCRAS,EMilman-RoleOfConvexity}.
Given a measure-metric space $(\Omega,d,\mu)$, we denote by $D_{FM} = D_{FM}(\mu)$ the best constant $D$ in the following \emph{first-moment inequality}:
\[
 \int | f - med_\mu(f) | d\mu \leq \frac{1}{D} \;\;\;\;\; \forall \text{ $1$-Lipschitz function $f$} ~,
\]
where as usual, $med_\mu f$ denotes a median of $f$ with respect to $\mu$. Recall from Section \ref{sec:pre} that $D_{Con_p} = D_{Con_p}(\mu)$ is the \emph{$p$-exponential concentration constant}, i.e. the best constant $D$ so that:
\begin{equation} \label{eq:con-p-again}
 \K_{(\Omega,d,\mu)}(r) \geq -1 + D r^p \;\;\; \forall r \geq 0 ~.
\end{equation}
Recall also that $D_{Poin} = D_{Poin}(\mu)$ denotes the Poincar\'e constant, and that $D_{Iso_1} = D_{Iso_1}(\mu)$ denotes the \emph{exponential isoperimetric constant}, which by Remark \ref{rem:iso_p} is equivalent (up to universal constants) to the \emph{Cheeger constant}, i.e. the best constant $D$ in the following \emph{linear isoperimetric inequality}:
\[
 \tilde{\I}_{(\Omega,d,\mu)}(v) \geq D v \;\;\; \forall v \in [0,1/2] ~.
\]

The following theorem was proved in \cite{EMilman-RoleOfConvexity} (see also \cite{EMilmanGeometricApproachPartI} for a slightly stronger statement and simplified proof):
\begin{thm}[\cite{EMilman-RoleOfConvexity}] \label{thm:main-ROC}
If $(\Omega,d,\mu)$ satisfies our convexity assumptions then:
\[
 D_{Iso_1} \geq c_1 D_{FM} \geq c_2 D_{Con_1} \geq c_3 D_{Poin} \geq c_4 D_{Iso_1} ~,
\]
where $c_1,c_2,c_3,c_4 > 0$ are some numeric constants.
\end{thm}
\begin{rem} \label{rem:main-ROC-trivial}
All of the above inequalities except for the first hold without any additional convexity assumptions: the second is trivial and the others follow from (\ref{eq:hierarchy-2}) and the subsequent comments.
The inequality $D_{Iso_1} \geq c_2 D_{Con_1}$ under our convexity assumptions also follows from Theorem \ref{thm:main-equiv}.
\end{rem}

Theorem \ref{thm:main-ROC} was used in \cite{EMilman-RoleOfConvexity} to obtain stability results for $D_{Iso_1}(\mu)$ when $\mu$ undergoes a perturbation so that our convexity assumptions are preserved. The control on the perturbation was measured in terms of control over some notion of distance between the original measure $\mu_1$ and the perturbed measure $\mu_2$. Three distances were analyzed: $\norm{d\mu_1 / d\mu_2}_{L^\infty}$, $\norm{d\mu_2 / d\mu_1}_{L^\infty}$ and the total variation distance $d_{TV}(\mu_1,\mu_2) = \sup_{A \subset \Omega} |\mu_1(A) - \mu_2(A)|$. Although some essentially sharp estimates were obtained in \cite{EMilman-RoleOfConvexity},
there is a certain drawback in using any of the above distances, which was already mentioned in the beginning of this section.

In this subsection, we analyze the stability with respect to two new distances: the Wasserstein distance $W_1(\mu_1,\mu_2)$ and the relative entropies $H(\mu_1 | \mu_2)$ and $H(\mu_2 | \mu_1)$, which provide further flexibility over the previous distances. The idea is based on the following elementary:

\begin{lem} \label{lem:W1-FM}
Let $\mu_1$,$\mu_2$ denote two probability measures on a common metric space $(\Omega,d)$. Then:
\[
 \abs{ \frac{1}{D_{FM}(\mu_2)}  - \frac{1}{D_{FM}(\mu_1)} } \leq W_1(\mu_1,\mu_2) ~.
\]
\end{lem}
\begin{proof}
Let $f \in \F(\Omega,d)$ denote a $1$-Lipschitz function. Then:
\begin{eqnarray*}
& & \int_{\Omega} |f - med_{\mu_2} f | d\mu_2 \leq  \int_{\Omega} |f - med_{\mu_1} f | d\mu_2 \\
& \leq &  \int_{\Omega} |f - med_{\mu_1} f| (d\mu_2 - d\mu_1) + \int_{\Omega} |f - med_{\mu_1} f | d\mu_1 \leq
W_1(\mu_1,\mu_2) + \frac{1}{D_{FM}(\mu_1)} ~,
\end{eqnarray*}
where we have used the dual characterization (\ref{eq:MKR-duality}) of $W_1$ in the last inequality. Taking supremum on $f$ as above, and exchanging the roles of $\mu_1$ and $\mu_2$, the assertion immediately follows.
\end{proof}

Using Theorem \ref{thm:main-ROC} (and Remark \ref{rem:main-ROC-trivial}), we immediately deduce the following stability result with respect to the $1$-Wasserstein distance:

\begin{thm} \label{thm:W1-stability}
If $(\Omega,d,\mu_2)$ satisfies our convexity assumptions then:
\begin{eqnarray}
\nonumber D_{Iso_1}(\mu_2) & \geq & c_1 D_{FM}(\mu_2) \geq \frac{c_1 D_{FM}(\mu_1)}{1+D_{FM}(\mu_1) W_1(\mu_1,\mu_2)} \\
\label{eq:con-1} & \geq &  \frac{c_2 D_{Con_1}(\mu_1)}{1+c_2' D_{Con_1}(\mu_1) W_1(\mu_1,\mu_2)}
\geq \frac{c_3 D_{Iso_1}(\mu_1)}{1+c_3' D_{Iso_1}(\mu_1) W_1(\mu_1,\mu_2)} ~,
\end{eqnarray}
where $c_i,c_i'>0$ are some numeric constants.
\end{thm}

Exchanging the roles of $\mu_1$ and $\mu_2$, we conclude:

\begin{cor}
If $(\Omega,d,\mu_1)$ and $(\Omega,d,\mu_2)$ satisfy our convexity assumptions and $W_1(\mu_1,\mu_2) \leq C \min(1/D_{Iso_1}(\mu_1),1/D_{Iso_1}(\mu_2))$, then $D_{Iso_1}(\mu_1) \simeq_C D_{Iso_1}(\mu_2)$, where the constants implied by $\simeq_C$ depend linearly on $1+C$.
\end{cor}

In practice, it is convenient to estimate $W_1(\mu_1,\mu_2)$ by using the relative entropies $H(\mu_2 | \mu_1)$ or $H(\mu_1 | \mu_2)$. These two possibilities turn out to be rather different.

\medskip

By Corollary \ref{cor:conc-weak-TE}, there exists a universal constant $c>0$ so that for every $p \geq 1$:
\begin{equation} \label{eq:wTC-p}
c \; D_{Con_p}(\mu) W_1(\nu,\mu) \leq H(\nu | \mu)^{1/p} + 1 \;\;\; \forall \text{ probability measure } \nu ~.
\end{equation}
Plugging this into the estimate (\ref{eq:con-1}) of Theorem \ref{thm:W1-stability}, and using the obvious fact that $D_{Con_p} \leq D_{Con_1}$ for $p \geq 1$, we obtain:
\begin{thm} \label{thm:going-down-H}
If $(\Omega,d,\mu_2)$ satisfies our convexity assumptions then:
\[
D_{Iso_1}(\mu_2) \geq  \frac{c_2'' D_{Con_1}(\mu_1)}{1+ C H(\mu_2 | \mu_1)} \geq
 \frac{c_3'' D_{Iso_1}(\mu_1)}{1+ C H(\mu_2 | \mu_1)} ~,
\]
where $c_2'',c_3'',C>0$ are some numeric constants. Moreover, for any $p\geq 1$:
\[
 D_{Iso_1}(\mu_2) \geq  \frac{c_2'' D_{Con_p}(\mu_1)}{1+ C H(\mu_2 | \mu_1)^{1/p}} ~.
\]
\end{thm}

A different estimate is obtained from Lemma \ref{lem:W1-FM} and Theorem \ref{thm:main-ROC} by proceeding as above but reversing the roles of $\mu_1$ and $\mu_2$.
This time, we cannot use the weak $(1,1)$ Transport-Entropy inequality given by (\ref{eq:wTC-p}) as in the proof of Theorem \ref{thm:going-down-H} (the reader may want to check this), so we employ its tight equivalent form given by Proposition \ref{prop:con-TE}:
\begin{equation} \label{eq:TC-ql}
c D_{Con_1}(\mu) W_1(\nu,\mu) \leq \varphi_{1}^{-1}(H(\nu | \mu)) \;\;\;  \forall \text{ probability measure } \nu ~,
\end{equation}
where, recall, $\varphi_1$ is given by (\ref{eq:def-varphi}).

\begin{thm} \label{thm:going-up-H}
There exists a universal constant $c>0$ so that if $H(\mu_1 | \mu_2) \leq c$ and $(\Omega,d,\mu_2)$ satisfies our convexity assumptions, then:
\[
D_{Iso_1}(\mu_2) \geq c_1 D_{Con_1}(\mu_2) \geq c_2 D_{FM}(\mu_1) \geq c_3 D_{Iso_1}(\mu_1) ~,
\]
where $c_1,c_2,c_3>0$ are some other universal constants.
\end{thm}
\begin{proof}
The first and last inequalities follow from Theorem \ref{thm:main-ROC} and Remark \ref{rem:main-ROC-trivial}. To deduce the middle inequality, we use Lemma \ref{lem:W1-FM} and (\ref{eq:TC-ql}):
\[
\frac{1}{D_{FM}(\mu_2)} - \frac{1}{D_{FM}(\mu_1)} \leq W_1(\mu_1,\mu_2) \leq \frac{\varphi_{1}^{-1}(H(\mu_1 | \mu_2))}{c D_{Con_1}(\mu_2)} ~.
\]
Multiplying by $D_{Con_1}(\mu_2)$ and using Theorem \ref{thm:main-ROC}, we conclude that:
\[
c' - \frac{D_{Con_1}(\mu_2)}{D_{FM}(\mu_1)} \leq \frac{1}{c} \varphi_{1}^{-1}(H(\mu_1 | \mu_2)) ~,
\]
for some universal constant $c'>0$. The assertion now clearly follows.
\end{proof}

A corollary which summarizes the resulting stability is:

\begin{cor}
If $(\Omega,d,\mu_1)$ and $(\Omega,d,\mu_2)$ satisfy our convexity assumptions then:
\begin{multline*}
C' \min\brac{\frac{1}{(1-C H(\mu_2 | \mu_1))_+} , 1+ C H(\mu_1 | \mu_2)} \geq \frac{D_{Iso_1}(\mu_2)}{D_{Iso_1}(\mu_1)} \\
\geq c' \max\brac{(1-C H(\mu_1 | \mu_2))_+ , \frac{1}{1+ C H(\mu_2 | \mu_1)}} ~,
\end{multline*}
where $c',C',C>0$ are numeric constants. In particular:
\begin{itemize}
\item
If $\min(H(\mu_2 | \mu_1) , H(\mu_1 | \mu_2)) \leq 1/(2C)$ then $D_{Iso_1}(\mu_1) \simeq D_{Iso_1}(\mu_2)$.
\item
If $A:=\max(H(\mu_2 | \mu_1) , H(\mu_1 | \mu_2)) < \infty$ then $D_{Iso_1}(\mu_1) \simeq_A D_{Iso_1}(\mu_2)$.
\end{itemize}
\end{cor}

\begin{rem}
Note that when $\min(H(\mu_1 | \mu_2),H(\mu_2 | \mu_1)) < 2$, the results of Theorems \ref{thm:going-down-H} and \ref{thm:going-up-H} may be recovered from our previous results from \cite{EMilman-RoleOfConvexity}, where the Total-Variation distance $d_{TV}$ was employed. This follows from the
well-known Pinsker-Csizsar-Kullback inequality \cite[Chapter 6]{Ledoux-Book}:
\begin{equation} \label{eq:Pinsker}
 d_{TV}(\mu_1,\mu_2) \leq \sqrt{\frac{1}{2} H(\mu_2 | \mu_1)} ~.
\end{equation}
Analogous stability results to those in this subsection were shown in \cite{EMilman-RoleOfConvexity} when $d_{TV}(\mu_1,\mu_2) \leq 1 - \eps$, for $\eps > 0$.
Since always $d_{TV}(\mu_1,\mu_2) \leq 1$, (\ref{eq:Pinsker}) suggests that when $\min(H(\mu_1 | \mu_2) , H(\mu_2 | \mu_1)) > 2$, we cannot formally obtain the results in this subsection from the previous ones in \cite{EMilman-RoleOfConvexity}.
\end{rem}

\begin{rem}
All the results in this subsection regarding $D_{Iso_1}$ also apply to the Poincar\'e constant $D_{Poin}$, since
these two are equivalent under our convexity assumptions (see Theorem \ref{thm:main-ROC} or Subsection \ref{subsec:pre-reverse}).
\end{rem}

\section{Equivalence between Transport-Entropy Inequalities with different cost-functions} \label{sec:TE-equiv}

Denote by $D_{TE_{\phi,\psi}}$ the best possible constant in the following $(\phi,\psi)$ TE inequality:
\[
\exists D >0 \;\;\; W_{c_{\phi,D}}(\nu,\mu) \leq \psi^{-1}(H(\nu | \mu)) \;\;\; \forall \text{ probability measure } \nu ~,
\]
where as usual $c_{\phi,D}$ denotes the cost-function $c_{\phi,D}(x,y) := \phi(D d(x,y))$. When $\phi$ ($\psi$) is the function $t^p$, we simply write $D_{TE_{p,\psi}}$ ($D_{TE_{\phi,p}}$) to be consistent with previous notation.

Using Jensen's inequality and the fact that $F_{p,s} := \varphi_p \circ \varphi_s^{-1}$ is convex when $2 \geq p \geq s \geq 1$, it is immediate to check that:
\begin{equation} \label{eq:TE-Jensen}
1 \leq s \leq p \leq 2 \Rightarrow  D_{TE_{1,\varphi_p}} \geq D_{TE_{\varphi_1,\varphi_p \circ \varphi_1^{-1}}} \geq D_{TE_{\varphi_s,\varphi_p \circ \varphi_s^{-1}}} \geq D_{TE_{\varphi_p,1}} ~.
\end{equation}
In view of the equivalence $D_{Poin} \simeq D_{TE_{\varphi_1,1}}$ stated in (\ref{eq:BGL-equiv}) and Proposition \ref{prop:con-TE}, the inequality between first and last terms in (\ref{eq:TE-Jensen}) for $p=1$ should be interpreted as a trivial proof of the well-known implication, due to Gromov and V. Milman \cite{GromovMilmanLevyFamilies}, that a Poincar\'e inequality always implies exponential concentration. In general, this shows that a $(\varphi_p,1)$ Transport-Entropy inequality implies $p$-exponential concentration (without relying on Marton's method). Similarly, when $s=1$, (\ref{eq:TE-Jensen}) should be interpreted as a trivial proof of the known implication that a $(\varphi_p,1)$ Transport-Entropy inequality ($p \in (1,2]$) implies the Poincar\'e inequality. This was shown in the case $p=2$ by Otto--Villani \cite{OttoVillaniHWI}, for $p=1,2$ by Bobkov--Gentil--Ledoux \cite{BobkovGentilLedoux}, and for $p \in [1,2]$ by Gentil--Guillin--Miclo \cite{GentilGuillinMiclo}; for further generalizations, see \cite[Theorem 22.28]{VillaniOldAndNew}.

\medskip

We conclude that Transport-Entropy inequalities provide a certain framework for deducing concentration and other inequalities, just by employing elementary tools such as Jensen's inequality. As a further application of Theorem \ref{thm:main-equiv}, we demonstrate in this section that under our various semi-convexity assumptions, reverse inequalities may be obtained, implying that TE inequalities with different cost-functions are in fact equivalent (up to universal constants) under these assumptions. Our procedure should by now be self-evident, so we omit the proofs.

\begin{thm} \label{thm:TE-equiv-2}
Assume that our convexity assumptions are satisfied ($\kappa=0$) for $(M,g,\mu)$. Then there exists
a universal constant $c>0$, so that for any $p \in [1,2]$, all of the constants in (\ref{eq:TE-Jensen}) are equivalent for all $s \in [1,p]$:
\[
D_{TE_{\varphi_p,1}} \geq c D_{TE_{1,\varphi_p}} ~.
\]
\end{thm}

\medskip

Similarly, Jensen's inequality trivially implies for $p \geq 2$ that:
\[
1 \leq s_1 \leq s_2  \;\;\; \Rightarrow \;\;\; D_{TE_{s_1,p}} \geq D_{TE_{s_2,p}} ~.
\]
The converse to this is addressed in the following:

\begin{thm} \label{thm:TE-equiv}
Assume that our $\kappa$-semi-convexity assumptions are satisfied for $(M,g,\mu)$ ($\kappa \geq 0$), and let $p\geq 2$. Then:
\begin{enumerate}
\item
If $\kappa=0$, then $(s,p)$ Transport-Entropy inequalities are equivalent for $s \in [1,p]$:
\[
D_{TE_{p,p}} \geq c^1_p D_{TE_{1,p}} ~.
\]
\item
If $\kappa > 0$ and $p>2$, then $(s,p)$ Transport-Entropy inequalities are equivalent for $s \in [1,p]$:
\[
D_{TE_{p,p}} \geq c^2(D_{TE_{1,p}},p,\kappa)  ~.
\]
\item
If $\kappa > 0$, then a strong-enough $(1,2)$ Transport-Entropy inequality implies a $(2,2)$ Transport-Entropy inequality:
\[
D_{TE_{1,2}} > \sqrt{\kappa/2} \;\; \Rightarrow \;\; D_{TE_{2,2}} \geq c^3(D_{TE_{1,2}},\kappa)  ~.
\]
\end{enumerate}
Here $c^1_{p}>0$ is a constant depending solely on its argument, and $c^2(\Delta,p,\kappa),c^3(\Delta,\kappa)$ are functions depending solely on their arguments, which in addition are strictly positive as soon as $\Delta$ is.
\end{thm}

%\setlinespacing{1.0}
\setlength{\bibspacing}{2pt}

\bibliographystyle{plain}

\def\cprime{$'$}

\end{document}